\documentclass[11pt]{amsart}
\usepackage{multirow,latexsym,mathptmx,graphicx,fancyhdr, enumerate}
\usepackage{url}
\usepackage{mathrsfs}
\usepackage{enumitem}
\usepackage{color} 
\usepackage[colorlinks,linkcolor=blue,anchorcolor=blue,citecolor=blue,backref=page]{hyperref}
\usepackage{graphics,epsfig}
\usepackage{graphicx} 
\usepackage{float}
\usepackage{epstopdf}
\usepackage[utf8]{inputenc}
\usepackage{hyperref} 
\usepackage{calc, amscd}
\usepackage[T1]{fontenc} 
\usepackage{pdfpages}
\usepackage{mathtools} 
\usepackage{scalerel,stackengine}
\usepackage{caption}
\usepackage{mathptmx}
\usepackage[per-mode=symbol]{siunitx}
\usepackage{bm}
\usepackage{amsmath}
\hypersetup{breaklinks=true}
\usepackage{amsthm,mathrsfs, multirow,xcolor,lscape,longtable,pbox,lipsum}
\usepackage{amsfonts,amscd,amssymb}
\usepackage{tikz-cd}
\usepackage[thinlines]{easytable}
\usepackage{microtype}
\usepackage{todonotes}
\usepackage{phaistos}
\usepackage{tabularx}

\usepackage[OT2,T1]{fontenc}
\DeclareSymbolFont{cyrletters}{OT2}{wncyr}{m}{n}
\DeclareMathSymbol{\Sha}{\mathalpha}{cyrletters}{"58}

\newtheorem{theorem}{Theorem}
\newtheorem{lemma}{Lemma}

\newtheorem{Conjecture}{Conjecture}

\newtheorem{definition}{Definition}

\newtheorem{remark}{Remark}

\newtheorem*{theorem*}{Theorem}
\newtheorem*{question*}{Question}

\usepackage{mathtools}

\newcommand{\Z}{{\mathbb{Z}}}
\newcommand{\E}{{\mathcal{E}}}

\newcommand{\F}{{\mathbb F}}
\DeclareMathOperator{\rank}{rank}

\newcommand\remove[1]{}

\DeclareMathOperator{\Sel}{Sel}

\begin{document}
\date{\today}
\author{Subham Bhakta} 
\address{Indian Institute of Science Education and Research, Thiruvananthapuram, India} \email{subham1729@iisertvm.ac.in}
\author{Srilakshmi Krishnamoorthy}
\address{Indian Institute of Science Education and Research, Thiruvananthapuram, India}
\email{srilakshmi@iisertvm.ac.in}


\newcommand\G{\mathbb{G}}
\newcommand\sO{\mathcal{O}}
\newcommand\sE{{\mathcal{E}}}
\newcommand\tE{{\mathcal{E}}}
\newcommand\sF{{\mathcal{F}}}
\newcommand\sG{{\mathcal{G}}}
\newcommand\sH{{\mathcal{H}}}
\newcommand\sN{{\mathcal{N}}}
\newcommand\GL{{\mathrm{GL}}}
\newcommand\HH{{\mathrm H}}
\newcommand\mM{{\mathrm M}}
\newcommand\fS{\mathfrak{S}}
\newcommand\fP{\mathfrak{P}}
\newcommand\fp{\mathfrak{p}}
\newcommand\fQ{\mathfrak{Q}}
\newcommand\Qbar{{\bar{\Q}}}
\newcommand\sQ{{\mathcal{Q}}}
\newcommand\sP{{\mathbb{P}}}
\newcommand{\Q}{\mathbb{Q}}
\newcommand{\tH}{\mathbb{H}}
\newcommand{\R}{\mathbb{R}}
\newcommand\Gal{{\mathrm {Gal}}}
\newcommand\SL{{\mathrm {SL}}}
\newcommand\Hom{{\mathrm {Hom}}}
\newtheorem{thm}{Theorem}[section]
\newtheorem{ack}[thm]{Acknowledgement}
\newtheorem{cor}[thm]{Corollary}
\newtheorem{conj}[thm]{Conjecture}
\newtheorem{prop}[thm]{Proposition}
\theoremstyle{definition}
\newtheorem{claim}[thm]{Claim}
\theoremstyle{remark}
\newtheorem*{fact}{Fact} 
\title[Watkins's conjecture for elliptic curves with a rational torsion]{Watkins's conjecture for elliptic curves with a rational torsion}
\begin{abstract}
Watkins's conjecture suggests that for an elliptic curve $E/\Q$, the rank of the group $E(\Q)$ of rational points is bounded above by $\nu_2 (m_E)$, where $m_E$ is the modular degree associated with $E$. In \cite[Thm.~2]{BKP23}, it has been established that Watkins's conjecture holds on average. This article investigates the conjecture over certain thin families of elliptic curves. For example, for prime $\ell$, we quantify the elliptic curves featuring a rational $\ell$-torsion that satisfies Watkins's conjecture. Additionally, the study extends to a broader context, investigating the inequality $\mathrm{rank}(E(\Q))+M\leq \nu_2(m_E)$ for any positive integer $M$.
\end{abstract}
\maketitle 
\tableofcontents
\section{Introduction}
Let $X_0(N)$ be the modular curve that can be realized as the Riemann surface obtained by the standard action of the congruence subgroup $\Gamma_0(N)$ on the upper half place $\mathbb{H}$. Consider $E/\Q$ to be any arbitrary elliptic curve. We know that $E$ is modular by the work of Wiles, Taylor-Wiles, et al. In other words, there exists a surjective morphism (defined over $\Q$)
$\phi_{E}:X_0(N) \to E$, which is called the \textit{modular parametrization} of $E$. Throughout the article, we shall assume that $\phi_E$ has a minimal degree, sending the cusp $\infty$ of $X_0(N)$ to the identity of $E$. We denote the minimal degree as $m_E$, the modular degree of $E$. In this article, we mainly aim to study Watkins's conjecture.
\begin{Conjecture}[Watkins]
    For any elliptic curve $E/\Q$, we have $\mathrm{rank}_{\Z}(E(\Q))\leq \nu_2(m_E)$.
\end{Conjecture}
The conjecture is known to be true when the modular degree is odd. This follows from the work of Calegari, Emerton ~\cite{CE09}, Yazdani~\cite{yazdani11} and Kazalicki, Kohem~\cite{KK18watkins}. When $E(\Q)[2]$ is trivial, Dummigan~\cite{neil06} heuristically proved that certain classes of elliptic curves satisfy Watkins's conjecture, under the assumption that $R\cong T$. This assumption is about a relation between the tangent space of a suitable deformation ring that captures relevant properties of the residual representation $\overline{\rho}:\mathrm{Gal}(\overline{\Q}/\Q)\to \mathrm{GL}_2(E[2])$, and a certain Selmer group which is large enough to capture the weak Mordell-Weil group.

An alternative heuristic proof of Watkins's conjecture is presented in \cite[Thm.~2]{BKP23}, demonstrating its average truth. Specifically, the conjecture holds true for nearly all elliptic curves under the Minimalist conjecture. This follows as a consequence of the primary outcome in Dummigan and Krishnamoorthy~\cite{DK13}, which essentially establishes a lower bound on $\nu_2(m_E)$ based on the number of prime factors in the conductor of the elliptic curve $E$. Recently Park, Poonen, Voight, and Wood \cite{rank21} conjectured that, there are only finitely many elliptic curves of rank $21$. With these heuristics, Watkins's conjecture is true for all but finitely many elliptic curves with a large number of bad reductions.

It is noteworthy that any elliptic curve over $\Q$ can be expressed in the minimal Weierstrass form 
\begin{equation}\label{eqn:Emin}
E_{A, B}: y^2 = x^3 + Ax + B,~A, B \in \Z,~p^4|A \implies p^6\nmid B.
\end{equation}
In the context of \cite[Thm.~2]{BKP23}, elliptic curves are arranged based on their height, defined as 
\begin{equation}\label{eqn:height}
H(E_{A,B}) = \max\Big\{|A|^3,|B|^2\Big\}^{1/6}.
\end{equation}
Throughout the article, we shall arrange the elliptic curves with respect to the height $H$, as defined above. Note that there are approximately $\asymp X^{5}$ elliptic curves with heights not exceeding $X$. The arithmetic statistics presented in \cite[Thm.~2]{BKP23} do not address thin families, i.e., families that exclusively consist of $0\%$ of all elliptic curves. In this article, our main focus revolves around the thin family given by the elliptic curves equipped with at least one rational torsion point. In other words, we work over the families of elliptic curves for which $G=E(\Q)_{\mathrm{tor}}$ is non-trivial. Due to the celebrated theorem of Mazur~\cite{Mazur1977}, it is enough to study the following $14$ possibilities of $G$: 
\begin{equation}
\label{mazur}
\begin{aligned}
\Z/n\Z & \quad \textrm{with $2 \le n \le 10$ or $n=12$} \\
\Z/2\Z \times \Z/n\Z & \quad \textrm{with $n=2,4,6,8$.}
\end{aligned}
\end{equation}
Note that Caro and Pasten \cite{CP22}, later Pasupulati and both of the authors of this article in \cite[Thm.~6]{BKP23}, exhibited a family of elliptic curves satisfying Watkins's conjecture when $G$ is one of the $10$ groups that has an even order. Following this, we set the following notations, that will be followed throughout the whole article 
\begin{definition}\label{def:main}
For any prime $\ell \in \{2,3,5,7\}$, we set
    $$\mathcal{E}=\{E_{A,B}~\mathrm{minimal}: A,B\in \Z^2\},~\mathcal{E}_{\ell}=\{E\in \mathcal{E}:E(\Q)[\ell]\neq~\mathrm{trivial}\},$$
and for the counting purpose, we set
    $$\mathcal{E}(X)=\{E\in \mathcal{E}: H(E)\leq X\},~\mathcal{E}_{\ell}(X)=\{E\in \mathcal{E}: H(E)\leq X,~E(\Q)[\ell]\neq~\mathrm{trivial}\},$$
i.e., $\mathcal{E}_{\ell}(X)=\mathcal{E}(X)\cap \mathcal{E}_{\ell}$. 
\end{definition} 
In this article, we aim to quantify the proportion of elliptic curves over each of the families in Definition~\ref{def:main}, that satisfy Watkins's conjecture. In this regard, let us first denote $\mathcal{E}'_2$ be the set of all elliptic curves $E_{a,b}:=y^2=x^3+ax^2+bx,~a,b\in \Z$. Note that any $E_{A,B}\in \mathcal{E}_2$ can be written in the Weierstrass form $E_{-3b,~3b^2+a} := y^2 = x^3 - 3bx^2 + (3b^2+a)x,~a, b \in \Z$. This shows that $\mathcal{E}_2 \subset \mathcal{E}'_2$. We shall arrange the elliptic curves in $\mathcal{E}'_2$ with respect to the function $H_2(E_{a,b})=\max\{|a|^2,|b|\}^{1/2}$. We shall demonstrate in Section~\ref{sec:E2setup} that the number of elliptic curves in $\mathcal{E}_2$ of height (with respect to $H$) at most $X$, is asymptotically the number of elliptic curves in $\mathcal{E}'_2$ of height (with respect to $H_2$) at most $X$. In this regard, we first prove the following.
\begin{theorem}\label{thm:uncond}
When arranging the elliptic curves in $\mathcal{E}'_2$ according to the height $H_2$, then Watkins's conjecture is true for a positive proportion of elliptic curves in $\mathcal{E}'_{2}$. Consequently, Watkins's conjecture is true for a positive proportion of elliptic curves in $\mathcal{E}_2$, arranging with respect to height $H$.
\end{theorem}
Moreover, we leverage the reasoning outlined in \cite{JMR} to construct a specific sub-family within $\mathcal{E}_2$ by considering quadratic twists of an elliptic curve, all of which confirms the Watkins's conjecture. This finding is documented in Corollary~\ref{cor:another}, and can be viewed as extensions of \cite[Thm.~1.2]{EP21} and \cite[Thm.~1.2]{caro2022watkins} for a specific family of elliptic curves.

Additionally, we will delve into an extended version of Watkins's conjecture for further investigation.
\begin{definition}
    For any integer $M\geq 0$, we say that an elliptic curve $E$ is $M$-Watkins, if $r+M\leq \nu_2(m_E).$
\end{definition}
Certainly, Watkins's conjecture suggests that any elliptic curve $E/\mathbb{Q}$ is $0$-Watkins. In \cite{BKP23}, several results regarding the counting of $0$-Watkins elliptic curves were discussed. The main objective of the article is to enumerate the generalized $M$-Watkins elliptic curves for any integer $M$. In this setting, we have the following.
\begin{theorem}\label{thm:e2mwatkins}
    For any integer $M$, a positive proportion of elliptic curves in $\mathcal{E}_2$ is $M$-Watkins when arranged with respect to height $H$.
\end{theorem}
Certainly, one can derive Theorem~\ref{thm:uncond} from Theorem~\ref{thm:e2mwatkins}. However, the density we achieve in this case is reduced by approximately a factor of $\frac{1}{(2M)^M}$, as demonstrated in Lemma~\ref{cor:Sboundprop}.

Subsequently, we investigate the $M$-Watkins property over $\mathcal{E}_3$ in both unconditional and conditional ways. For the unconditional case, the key ingredient lies in studying the elliptic curves of \textit{Type-I}. Specifically, this means examining the elliptic curves of the form $E_a:y^2=x^3+a,~a\in \Z$. In this context, we prove the following result.
\begin{theorem}\label{thm:x3+k}
For the \textit{Type-I} elliptic curves, we have the following estimates for any integer $M$:
\begin{enumerate}
\item [(i)] almost all the elliptic curves over Type-I are $M$ Watkins. More precisely, 
$$\#\Big\{a\in \Z: |a|\leq X,~E_a~\mathrm{is~not}~M-\mathrm{Watkins}\Big\}=O_M\left( \frac{X}{\log \log X}\right).$$
   \item[(ii)] For any integer $k$, consider the set of integers $N_{k,m}=\{kn^2:n\in \Z\}$. Then, for any integer $M$, we have $\#\Big\{a\in N_{k,m}:|a|\leq X,~E_a~\mathrm{is}~M-\mathrm{Watkins}\Big\}\gg_{k,M}\left(\frac{X}{\log X}\right)^{1/2}$.\\
    \item[(iii)] The size of $\mathcal{E}_2(X)$ is approximately $\asymp X^3$, and the number of $M$-Watkins \text{Type-I} elliptic curves in $\mathcal{E}_2(X)$ has a lower bound of order $\gg \frac{X}{(\log X)^{1/4}}$.
\end{enumerate}
\end{theorem} 
As an immediate application, we deduce the following consequence for $\mathcal{E}_3$, taking $k=1$.
\begin{cor}\label{thm:3-torsion}
For any integer $M\geq 0$, infinitely many elliptic curves in $\mathcal{E}_{3}$ are $M$-Watkins. More specifically, the size of $\mathcal{E}_3(X)$ is approximately $\asymp X^2$, and the number of $M$-Watkins elliptic curves in $\mathcal{E}_3(X)$ has a lower bound of order $ \frac{X}{(\log X)^{1/2}}$.
\end{cor}

Furthermore, we aim to enhance Theorem~\ref{thm:3-torsion} by proving that all elliptic curves within $\mathcal{E}_3$, $\mathcal{E}_5$, and $\mathcal{E}_7$ adhere to being $M$-Watkins, with $100\%$ certainty. The approach involves parametrizing the points in $\mathcal{E}_5$, and $\mathcal{E}_7$ using appropriate elliptic surfaces over $\Q(t)$. For the points in $\mathcal{E}_3$, it turns out that, \textit{except} for the \textit{Type-I} curves of the form $y^2=x^3+b^2,~b\in\Z$, all other points can also be parametrized by a suitable elliptic surface over $\Q(t)$. After addressing such \textit{Type-I} curves in the proof of part $(ii)$ in Theorem~\ref{thm:3-torsion} or Theorem~\ref{thm:x3+k}, we turn to studying the ranks and conductors and ranks of conductors of the elliptic curves, lying in a given $1$-parameter family of elliptic curves. Proposition~\ref{prop:maintool} is our main tool for studying the conductors, while for ranks, we rely on some of the \textit{standard} conjectures in the theory of elliptic curves mentioned in \cite{Silverman}. More specifically, we shall prove the following. 
\begin{theorem}\label{thm:isogeny}
The sizes of $\mathcal{E}_5(X)$ and $\mathcal{E}_7(X)$ are approximately $\asymp X$ and $\asymp X^{1/2}$, respectively. Assume that the conjectures $B, C$, and $D$ in \cite{Silverman} hold. Then for any integer $M$, and any prime $\ell \in \{3,5,7\}$, almost all the elliptic curves in $\mathcal{E}_{\ell}$ are $M$-Watkins.
\end{theorem}
\subsection{Overview and strategies}
Section~\ref{sec:tools} provides a brief overview of the rank bounds, highlighting the relationships between modular degrees and ranks. Specifically, Corollary~\ref{cor:E2} boils down the study of Watkins's conjecture to understanding the ranks and conductors of elliptic curves. To get good enough bounds on the ranks over $\mathcal{E}_2$ and $\mathcal{E}_3$, we employ classical methods such as Descent via $2$ and $3$-isogenies. Specifically, Corollary~\ref{cor:rankmain} provides a sufficiently nice bound in terms of the conductors of the elliptic curves over $\mathcal{E}_2$, and in Section~\ref{sec:bounding2rank}, we further refine this through a sieving process in Lemma~\ref{lem:-mbound}, demonstrating that the rank bound in Corollary~\ref{cor:rankmain} holds in almost all cases. To count the number of elliptic curves in $\mathcal{E}_2$, we use Lipschitz's lattice point counting criteria in Section~\ref{sec:E2setup}. Furthermore, to understand $M$-Watkins property over $\mathcal{E}_2$, we study the quadratic twists of $E_0:y^2=x^3-1$ in Section~\ref{sec:low2}, using the techniques outlined in \cite{JMR} to determine when the ranks are at most $1$. 

To study $M$-Watkins over $\mathcal{E}_3$, we employ both conditional and unconditional approaches. First, we count the number of elliptic curves in $\mathcal{E}_3$ in Section~\ref{sec:E3setup}. For the unconditional approach to $M$-Watkins, we rely on the rank bounds for the family $y^2=x^3+a,~a\in \Z$, using the classical Descent via $3$-isogeny, along with the main result of \cite{3-iso}, which shows that a typical elliptic curve of $\textit{Type-I}$ has a small rank. For the approach through $3$-isogeny, we follow the rank bound discussed in Section~\ref{sec:3descent}. Specifically, the observation at Lemma~\ref{lem:sqcase} produces a lot of quadratic fields, whose $3$-torsion rank of class group is significantly smaller than the number of prime factors of the discriminant. Finally, we accomplish our goal through the unconditional approach in Section~\ref{sec:we3}.

For the conditional approach over $\mathcal{E}_3$, which is also our approach towards $\mathcal{E}_5$ and $\mathcal{E}_7$, we rely on the existence of a universal elliptic curve over parametrizing them. We shall briefly discuss this in Section~\ref{sec:univ}. Now the heart of the study involves two steps: first, showing that the average rank of the elliptic curves lying in a non-split elliptic surface is bounded. To prove this, we need to rely on various classical conjectures and extend the main result of Silverman in \cite{Silverman}, as briefly discussed in Section~\ref{sec:avgrank}. The second step involves studying the typical number of prime factors of the square-free parts of polynomial values, which we address in Proposition~\ref{prop:maintool}. Note the normal order of $\omega(n)$ is $\log \log n$, and this normal order is consistent for both $\omega(n)$ and $\omega(s(n))$. Following this principle, we extend the normal order of polynomial values from \cite[Thm.~3.2.3]{Murty-Alina}. The entire conditional approach is carried out in Section~\ref{sec:higher}.
\subsection{Notations}
We use the symbol $X$ as a quantifying parameter and $x$ for one of the coordinates of an elliptic curve. The symbol $r$ is used for rationals, and $t$ serves as an indeterminate for the polynomial rings. 

By $O_{S_1, S_2,\cdots, S_k}(B)$ we mean a quantity with absolute value at most $cB$ for some positive constant positive constant $c$ depending on $S_1,\cdots, S_K$ only; if the subscripts are omitted, the implied constant is absolute. We write $A\ll_{S_1,\cdots,S_K} B$ for $A =O_{S_1, S_2,\cdots, S_k}(B)$. We write $A=o(B)$ for $A/B\to 0$, and $A\asymp B$ for $A/B\to c$, for some absolute constant $c$.

Any natural number $n$ can be written as $n=n'\times \Box$, where $n'$ is the square-free part of $n$ with $(n',\Box)=1$. Then, we denote $s(n)$ to be the square-free part of $n$. For any integer $a$, we denote $\omega(a)$ and $s(a)$ respectively be the number of distinct prime factors, and the square-free part of $|a|$. Additionally, use the standard notation $\mu(a)$ to indicate the Möbius function of $|a|$. For any rational number $r=\frac{a}{b}$ in the lowest terms, we set $\omega(r)=\omega(a)+\omega(b)$, and $s(r)$ as $s(ab)$. Furthermore, for any finite set of primes $S$, we denote $\omega_{S}(r)$ to be the number of prime factors of $r$, that are not in $S$. Moreover, for any integer $n$, we denote $
\mathbb{P}^1(n) = \mathbb{P}^1(\mathbb{Z}/n\mathbb{Z})$, the projective line over the ring of integers modulo $n$. 

Given an elliptic curve $E/\Q$, we denote $\Delta(E)$ and $N(E)$ as the discriminant and conductor of $E$, respectively. Moreover, we denote $H(E)$ to be the height of the elliptic curves as in \ref{eqn:height}.

\subsection*{Acknowledgements}
We extend our sincere gratitude to IISER Thiruvananthapuram for providing an excellent research environment that significantly contributed to the success of this project. We are deeply thankful to the institute for their invaluable support. SB is supported by the institute fellowship of IISER Thiruvananthapuram. 

A part of this study was conducted during the workshop \textit{Rational points on Modular curves} at the International Centre for Theoretical Sciences. Furthermore, some parts of study was also conducted during the first author's visit to the Chennai Mathematical Institute, Mathematisches Forschungsinstitut Oberwolfach, and the University of Kiel. We are grateful to these institutions for their exceptional hospitality and support.

We would like to acknowledge CS Rajan for posing a thought-provoking question during the second author's presentation at ICTS, which inspired us to delve into the study of $M$-Watkins elliptic curves.

Our appreciation also extends to Loïc Merel and Sunil Kumar Pasupulati for their insightful discussions and contributions to this work.

\section{Tools with modular degrees and ranks}\label{sec:tools}
In this section, we will discuss the necessary materials to help count the elliptic curves satisfying Watkins's conjecture. 

\subsection{Dummigan, Krishnamoorthy's lower bound}\label{sec:DK}
For each $p^{\alpha}\mid \mid N$, consider the Atkin-Lehner involution $W_p$ on $X_0(N)$. The set $\left\{W_p~:~d\mid \mid N\right\}$ forms an abelian subgroup of automorphisms of rank equal to $\omega(N)$. For every  prime divisor $p \mid N$, there exists $w_p(f_E)\in \{\pm 1\}$ such that $W_p(f_E)=w_p(f_E) f_E$. 

We now consider the homomorphism $(\Z/2\Z)^{\omega(N)} \to \{\pm 1\}$ defined by, $W_d \mapsto w_d$. Let $W'$ denote the kernel of this map. According to Dummigan and Krishnamoorthy \cite[Prop.~2.1]{DK13}, there exists a homomorphism $W'\to E(\Q)[2]$ with kernel $W''$, and $m_E$ divides the order of $W''$.

The modular parametrization gives a map $\pi: J_0(n) \to E$. The key observation here is that if $w\in W'$, then $\pi(w(D))=\pi(D)$ for any $D\in J_0(N)$. Taking $P_0\in X_0(N)(\Q)$ and $D=P_0-\omega(P_0)\in J_0(N)$, we have $\pi(w(D))\in E(\Q)[2]$. This yields a map $W'\to E[2]$ with kernel $W''$. Consequently, considering $X''=X_0(N)/W''$ and $J''=J(X'')$, the map $\pi$ induces a map $\pi'':J'' \to E$ such that $\deg(\pi)=\# W''\deg(\pi'')$. Following this, Dummigan and Krishnamoorthy \cite[Prop.~2.1]{DK13} deduced the following estimate.
\begin{align*}
    \nu_2(m_E)\geq \nu_2\left(\frac{\#W'}{E(\Q)[2]}\right)
    &=\nu_2\left({\#W'}\right) -\nu_2\left({E(\Q)[2]}\right)\\
&=\omega(N(E))- \dim_{\Z/2\Z}(W/W')-\mathrm{dim}_{\mathbb{Z}/2\mathbb{Z}}(E(\mathbb{Q})[2]).
\end{align*}
In particular, we deduced the following, that will be used for the treatment over $\E_2$.
\begin{cor}\label{cor:E2}
   For any elliptic curve $E/\Q$ with $E(\Q)[2]=\Z/2\Z$, we have the following inequality
$$\nu_2(m_E) \geq \omega(N(E))-1-\dim_{\Z/2\Z}(W/W') \geq \omega(N(E))-2.$$
\end{cor}
\begin{remark}\rm\label{rem:imp}
Recall that for an elliptic curve $E_{A,B}$ in its minimal form as in (\ref{eqn:Emin}), the corresponding conductor $N(E_{A,B})$ is, up to some bounded factor of $2$ and $3$, the product of all primes $p$ dividing the discriminant $\Delta(E_{A,B})$. In light of Corollary~\ref{cor:E2}, our primary objective is to count the set of all such minimal $E_{A,B}$, for which $N(E_{A,B})$ is divisible by many large primes. Additionally, we also need to study the rank growth to study the $M$-Watkins elliptic curves.
\end{remark}
\subsection{Estimation of ranks over $\mathcal{E}_2$ and $\mathcal{E}_3$}
\subsection{Descent via 2-isogeny}\label{sec:2descent}
Let $E_{a,b}$ be of the form $y^2=x^3+ax^2+bx,~a,b\in \Z$, and $E'_{a,b}$ be the dual curve, defined by $y^2=x^3-2ax^2+(a^2-4b)x$. Now we briefly recall the method of descent via $2$-isogeny. Consider the spaces of the Homogeneous space of the form 
$C_{d_1,d_2,F}:Z^2 = d_1U^4 + F U^2V^2 + d_2V^4,$ and for any set of prime $T$ containing $\infty$, let us denote
$Q(T) := \{d \in \Q^{*}/(\Q^{*})^2 \mid \nu_p(d) = 0 \pmod 2,~\forall p \not\in T\}$.
Then, we have the following key bound from \cite[Eqn.~(5)]{RPC}:
   \begin{equation}\label{eqn:2rankbound-}
    \rank\left(E(\Q)\right) \leq \dim_{\F_2}\Sel^{\phi}(E)+ \dim_{\F_2}\Sel^{\hat{\phi}}(E')-2,
 \end{equation}
where the Selmer groups are precisely of the following forms.
\begin{equation}\label{eqn:1sphi}
\Sel^{\phi}(E_{a,b})=\{d \in Q(T_1) : C_{d,\frac{a^2-4b}{d},-2a}(\Q_p)\neq \phi,~\forall p \mid 2b(a^2-4b) \},
\end{equation}
\begin{equation}\label{eqn:1sphi'}
\Sel^{\hat{\phi}}(E'_{a,b})=\{d \in Q(T_2) : C_{d,\frac{b}{d},a}(\Q_p)\neq \phi,~\forall p \mid 2b(a^2-4b) \},
\end{equation}
where $T_1$ (resp. $T_2$) is the set of prime factors of $a^2-4b$ (resp. $b$), containing $\infty$. It follows from the proof of \cite[Lem.~2.1]{RPC} that, 
\begin{equation}\label{eqn:eitheror}
\Sel^{\phi}(E_{a,b})\subset Q(T_1\setminus \{\infty\}),~\mathrm{or}~\Sel^{\hat{\phi}}(E'_{a,b})\subset Q(T_2\setminus \{\infty\}).
\end{equation} 
With this observation, we deduce the following.
\begin{cor}\label{cor:rankmain}
    Let $a,b$ be any two co-prime integers for which $\Sel^{\phi}(E_{a,b})\subset Q(T\setminus \{p\})$ for some finite prime $p$. Then, we have
   \begin{equation}\label{eqn:2rankbound}
   \rank\left(E_{a,b}(\Q)\right) \leq \omega(N(E_{a,b}))-2.
\end{equation} 
\end{cor} 
\begin{proof}
Given that $\Sel^{\phi}(E_{a,b})\subset Q(T\setminus \{p\})$ for some finite prime $p$, it follows immediately from equations (\ref{eqn:2rankbound-}) and (\ref{eqn:eitheror}) that $\rank\left(E_{a,b}(\Q)\right) \leq \omega(a^2-4b)+\omega(b)-1$. Since $a,b$ are co-prime, we also have $\omega(b(a^2-4b))=\omega(b)+\omega(a^2-4b)$. In other words, $\rank\left(E_{a,b}(\Q)\right) \leq \omega(b^2(a^2-4b))-1$. On the other hand, $\Delta(E_{a,b})=16b^2(a^2-4b)$, which shows that $\omega(b^2(a^2-4b))\leq \omega(N(E_{a,b}))$. This completes the proof.
\end{proof}

\subsection{Descent via $3$-isogeny}\label{sec:3descent}
Let $E/\Q$ be an elliptic curve, equipped with $\phi: E \to E'$ a rational isogeny of degree
$3$, with the dual isogeny $\hat{\phi}:E'\to E$ satisfying $\phi \circ \hat{\phi}=[3]$. Arguing similarly as in \cite[Secc.~2]{RPC} for the $3$-isogeny $\phi$, we get 
   \begin{equation}\label{eqn:3rankbound}
    \rank\left(E(\Q)\right) \leq \dim_{\F_3}\Sel^{\phi}(E)+ \dim_{\F_3}\Sel^{\hat{\phi}}(E').
 \end{equation}
To understand the Selmer groups in (\ref{eqn:3rankbound}), let us mention \cite[Sec.~1]{JMS} that the equation of $E$ can be expressed either in the form of a \textit{Type-I curve} $E_a : y^2 = x^3 + a$, or of a \textit{Type-II curve} $E_{a,b} : y^2 = x^3 + a(x-b)^2$. In this article, we will focus on the \textit{Type-I} case due to its computational simplicity. 
 
\subsubsection{Treatment for Type-I}\label{sec:type-I}
We follow the discussion from \cite[Sec.~1.1]{JMS}. Consider the elliptic curve denoted by the equation $E_a: y^3 = x^2 + a$. Notably, we observe that the discriminant of $E_a$ is $\Delta(E_a) = -2^4 \cdot 3^2 \cdot a^2$, and its $j$-invariant is $j(E_a) = 0$. Additionally, the curve possesses complex multiplication by the field $\mathbb{Q}(\sqrt{-3})$.
The set $C=\{ O, (0,\sqrt{a}), (0,-\sqrt{a})\}$ is a subgroup of $E(\overline{\Q})$ of order $3$. Let us then consider the $3$-isogeny $\phi:E_{a}\to E'_a$, where $E'_a=E_{-27a}$, and denote $\hat{\phi}_a$ be the corresponding dual isogeny $\hat{\phi}_a:E'_a\to E_a$. Let us set $K_a$ to be the number field $\Q(\sqrt{-3a})$, and denote $\mathcal{O}_{K_a}$ to be its ring of integers.

\begin{definition}
Let $S$ be any finite set of finite primes of $\mathcal{O}_{K_a}$. We define \begin{align*}
    H(S)=\Big\{x \in K_a^{\times}/\left(K^{\times}_a\right)^3 \Big \vert \nu_p(x)\equiv 0\pmod{3},~\text{for any} ~p\notin S  \Big\}.
\end{align*}
\end{definition}
From \cite[Lemm.~3.4]{BA04}, we have 
\begin{equation}\label{eqn:sbound}
\dim_{\F_3} \left(H(S)\right) \leq r_3(K_a)+\dim_{\F_3}\left( U_{K_a}/U_{K_a}^3\right)+\# S, 
\end{equation}
where $r_3(K)$ denotes the $3$-rank of the ideal class group of $K$, and $U_K$ denotes the group of units of $K$. 

For any integer $a$, consider $S_a$ be the set of primes defined by: $2,3\in S_a$, and a prime $p\neq 2,3$ is in $S_a$ if and only if, $\nu_p(a)=2,4 ~ \text{and} ~ \left(\frac{-3}{p}\right)=1$. Combining (\ref{eqn:sbound}) with \cite[Thm.~3.5]{BA08}, we have 
\begin{equation}\label{eqn:sphi}
\dim_{\F_3}(\Sel^{\phi}(E_a))\leq r_3(K_a)+\dim_{\F_3}\left( U_{K_a}/U_{K_a}^3\right)+\# S_a.
\end{equation}
Furthermore, arguing the same for the dual $E'_a:=E_{-27a}$, we have the bound
\begin{equation}\label{eqn:sphi'}
\dim_{\F_3}(\Sel^{\hat{\phi}}(E'_a))\leq r_3(K_{-27a})+\dim_{\F_3}\left( U_{K_{-27a}}/U_{K_{-27a}}^3\right)+\#S_{-27a}.
\end{equation}
\subsubsection{3-torsion ideal classes of quadratic fields} To bound the Selmer ranks in (\ref{eqn:sphi}) and (\ref{eqn:sphi'}), a key requirement is to establish bounds on the $3$-torsion ranks of the class groups of quadratic fields. We denote $h_3(K_a)$ to be the size of the $3$-torsion part. Of course, $r_3(a)=\log_3(h_3(K_a))$. It is conjectured \cite[Conj.~1]{Ellen07} that $h_3(K_a)=O_{\varepsilon} (|a|^{\varepsilon})$, or equivalently, $r_3(K_a)=O_{\varepsilon}(\log |a|)$, for any $\varepsilon>0$. For the bound on $h_3$, Helfgott-Venkatesh~\cite[Thm.~4.2]{HV} obtained the exponent $0.45$, and later, Ellenberg and Venkatesh improved it to $|a|^{\frac{1}{6}+\varepsilon}$. However, the pointwise bound is not sufficient for us, as we need to focus on $a$ for which $r_3(K_a)=O(\omega(a))$. Since $\omega(a)$ is typically of order $\log \log |a|$, we aim for a stronger phenomenon here. It is perhaps worth noting that Kishi and Miyake in \cite{Miyake} provided criteria for $r_3(K_a)=1$, though this appears challenging for quantitative analysis. One thing is to note that, $r_3(K_a)$ is just $O(1)$ over the \textit{almost} squares $a$. For the general case, while Heilbronn, Davenport~\cite{HD} would suffice for our purposes, we will instead use \cite{3-iso} regarding the average rank of elliptic curves of \textit{Type-I}.

\section{Setting up the ground work for counting with heights}\label{sec:setup}
For any pair $(A,B)\in \Z^2$, denote $E_{A,B}$ be the elliptic curve $y^2=x^3+Ax+B$. Recall that we are arranging them with respect to the height function $H(E_{A,B})=\max\left\{|A|^3,|B|^2\right\}^{1/6}$, and studying the proportion of all such $(A,B)$, for which $E_{A,B}$ satisfies the Watkins's conjecture. Note that there are $\asymp X^5$ many elliptic curves of height at most $X$. For each $\ell\in \{2,3,5,7\}$, the main aim of this section is to discuss how many of the elliptic curves with height at most $X$ have $\ell$-torsion. In Section~\ref{sec:2torsion} and Section~\ref{sec:higher}, we will further investigate how often these curves satisfy Watkins's conjecture, or in general, the $M$-Watkins property.

\subsection{Counting over $\E_2$}\label{sec:E2setup}

We know \cite[Lem.~5.1, part (a)]{HS17} that an elliptic curve $E_{A,B}$ has a rational point of order $2$ if and only if there exists $a,b \in \Z$ such that $A=a, B=b^3+ab$. This leads us to consider the lattice points in the region $R_2(X)=\Big\{(a,b)\in\R^2:|a|<X^2\text{ and }|b^3+ab|<X^3\Big\}$. In other words, we have $\mathcal{E}_2(X)=R_2(X)\cap \Z^2$. By the principle of Lipschitz \cite{DA}, we have 
\begin{equation}\label{eqn:E2}
\#\E_2(X)=\# (R_2(X)\cap \Z^2)\sim \mathrm{Vol}(R_2(1))X^3.
\end{equation}
Moreover, \cite[Lem.~5.2]{HS17} shows that where the numerical value of $\mathrm{Vol}(R_2(1))$ is precisely given by 
\begin{equation}\label{eqn:vol}
2\log(\alpha_{-}/\alpha_{+}) + \frac{4}{3}(\alpha_{+} + \alpha_{-}),~\mathrm{where~\alpha_{\pm}~is~the~unique~real~root~of}~x^3 \pm x - 1.
\end{equation}

In this section, let us denote $\mathcal{E}'_2$ be the set of all elliptic curves $E_{a,b}:=y^2=x^3+ax^2+bx$. First of all, let us note that any $E_{A,B}\in \mathcal{E}_2$ can be written in the Weierstrass form $E_{-3b,~3b^2+a} := y^2 = x^3 - 3bx^2 + (3b^2+a)x,~a, b \in \Z$. This shows that $\mathcal{E}_2 \subset \mathcal{E}'_2$. Let us arrange the elliptic curves in $\mathcal{E}'_2$, with respect to the natural height function defined by $H_2(E_{a,b})=\max\left\{|a|^2,|b|\right\}^{1/2}$. With this, we set 
$$R'_2(X)=\Big\{(a,b)\in \R^2: |a|\leq X,~|b|\leq X^2\Big\},~\mathcal{E}'_2(X)=R'_2(X)\cap \Z^2.$$ 
Note that $H(E_{A,B})\leq X$ if $|a|\asymp X^2$ and $|b|\asymp X$, where $A=a,~B=b^3+ab$. In other words, $H(E_{A,B})\leq X$ if and only if $H_2(E_{-3b,3b^2+a})\asymp X$. In particular, the number of elements in $\mathcal{E}_2(X)$ and
$\mathcal{E}'_2 (X)$ are asymptotically comparable to $X^3$, i.e., $\#\mathcal{E}_2(X)\asymp \#\mathcal{E}'_2(X)\asymp X^3$. 

Furthermore, let us consider  $\Z^2_{1}=\{(A,B): \mathrm{gcd}(A,B)=1\}$, and set
\begin{equation*}\label{eqn:coprime}
    \mathcal{E}_{1,2}(X)=\mathcal{E}_2(X)\cap \Z^2_1,~\mathcal{E}'_{1,2}(X)=\mathcal{E}_2(X)\cap \Z^2_1.
\end{equation*}
The requirement of considering co-prime pairs arises due to Corollary~\ref{cor:rankmain}. Since we have $\#\mathcal{E}_{1,2}(X)\asymp \#\mathcal{E}'_{1,2}(X)\asymp X^3$, \textit{to conclude the proof of Theorem~\ref{thm:uncond}, it suffices to work with the set $\mathcal{E}'_{1,2}=\bigcup\limits_{X\in \R}\mathcal{E}_{1,2}(X)$.}

\subsection{Counting over $\E_3$}\label{sec:E3setup}
We know from \cite[Lem.~5.1, part (b)]{HS17} that $E_{A,B}$ has a rational point of order $3$ if and only if there exists $a, b \in \Z$ such that $A=6ab+27 a^4$ and $B=b^2-27a^6$. In the same spirit as before, we consider
$R_3(X)=\{(a,b)\in\R^2:|6ab+27a^4|<X^{2}\text{ and }|b^2-27a^6|<X^{3}\}$. Again, by the principle of Lipschitz \cite{DA}, we have 
\begin{equation}\label{eqn:E3}
\E_3(X)=\# (R_3(X)\cap \Z^2)\asymp X^2.
\end{equation}
To study Watkins's conjecture for the points in $\mathcal{E}_3$, we will utilize conjectures such as the Generalized Riemann Hypothesis (GRH). However, for unconditional analysis, we hall focus on the \textit{Type-I} family of elliptic curves given by $E_a: y^2 = x^3 + a$, as described briefly in Section~\ref{sec:type-I}. It turns out that such an $E_a$ is in $\E_3$ when $a$ is a square. We organize the elliptic curves $\{E_a\}$ based on their standard height $H(E_a)=|a|^{1/3}$. Notably, there are $8X^3$ many integers $a\in \Z$ satisfying $H(E_a)\leq X$.

\subsection{Universal family for $\E_3,\E_5,\E_7$}\label{sec:univ}
Once we study Watkins's conjecture over $\mathcal{E}_2$, i.e. the points briefly described in Section~\ref{sec:E2setup}, it is sufficient to focus only on cases where the torsion group is $\mathrm{G}=E(\Q)_{\mathrm{tor}}$ is non-trivial, and $2\nmid \# \mathrm{G}$. In the sense of \cite[Sec.~1.4.2]{HS17}, there is a universal elliptic curve $\mathcal{E}$ over an open subset of $\mathbb{A}^1$, which parametrizes the elliptic curves whose group of rational points contains $G$. Then the universal property \cite[Sec. 1.4.2]{HS17} gives an elliptic surface $\mathcal{E} \to \mathbb{A}^1$, capturing the required property with $\ell=3,5$ or $7$. Let us discuss it briefly, as we need know about the parametrizations.

Let $S$ be a scheme. An elliptic curve over $S$ is a map
$\pi : E \to S$, together with a section $\sigma : S \to E$, such that every fiber of $\pi$ is a smooth projective curve of genus $1$. We denote $Y_1(N)$ to be the curve, whose points are given by
$Y_1(N )(S) = \{(E/S, P/S)~|~N \cdot P = 0\},$
and consider $X_1(N)$ be the Deligne-Mumford compactification of $Y_1(N)$. It turns out that, $X_1(N)$ is a scheme for any $N>3$. Since $X_1(5), X_1(7)$ are of genus $0$, there exists an isomorphisms 
$\varepsilon_5:X_1(5)\to \mathbb{A}^1(\Q),~\varepsilon_7:X_1(7)\to \mathbb{A}^1(\Q)$. These isomorphisms also give rise to the polynomials $f_5,g_5,f_7,g_7\in \Z[t]$, parametrizing almost all the elements of $\mathcal{E}_5$ and $\mathcal{E}_7$. Combining the data in \cite[Table 2]{HS17} and \cite[Prop.~3.2]{HS17}, we have $\mathrm{deg}(f_5)=4,~\mathrm{deg}(g_5)=6$, and $\mathrm{deg}(f_7)=8,~\mathrm{deg}(g_7)=12$. In particular, it turns out from \cite[Table 1]{HS17} that 
\begin{equation}\label{eqn:E5,E7}
\mathcal{E}_{5}(X)\sim X,~\mathcal{E}_{7}(X)\sim X^{1/2}.
\end{equation}

On the other hand, even though $X_1(3)$ is not scheme, almost all the points $X_1(3)(\Q)$ can still be paramterized by polynomials $f_3,g_3$ with $\mathrm{deg}(f_3)=1,\mathrm{deg}(g_3)=2$. The exceptional points in $X_1(3)(\Q)$ that can not be parametrized, are $\textit{Type-I}$ curves of the form $y^2=x^3+a$, where $a$ is a square. For more details, the reader may refer to \cite[Sec.~3.3]{HS17}

\section{Watkins's conjecture over $\E_2$ and $\E_3$}\label{sec:2torsion}
Let us recall the discussions on rank bound from Section~\ref{sec:2descent}. Our objective is to examine the lack of local solubility within a specific family of homogeneous spaces. Throughout this section, we shall follow the terminologies as in Section~\ref{sec:E2setup}.
\subsection{Bounding Selmer ranks}\label{sec:bounding2rank}
In this section, we want to eliminate sufficiently many hall divisors $d$ of $a^2-4b$ for which $C_{d,\frac{a^2-4b}{d}}(\Q_p)=\phi$ for some prime $p\mid a^2-4b$. 
\begin{lemma}\label{lem:phisriteria}
Let $p > 3$ be a prime dividing $b$ such that $p\nmid a$ and $d$ be a square-free hall divisor of $a^2-4b$. Then $C_{d,\frac{a^2-4b}{d}}(\Q_p)= \phi$ if and only if the following two conditions hold.
\begin{itemize}
    \item[(i)] $\left(\frac{d}{p}\right)=-1$
    \item[(ii)] $\nu_p(b)=\mathrm{odd}~\mathrm{or}~\left(\frac{a}{p}\right)=1$
\end{itemize}
\end{lemma}
\begin{proof}
Our goal is to ensure that none of the conditions in part (1) of \cite[Lem.~3.1]{RPC} are satisfied.

Given that $p \mid b$, it is important to observe that $\left(\frac{d}{p}\right)=-1$ if and only if $\left(\frac{\frac{a^2-4b}{d}}{p}\right)=-1$. Due to the condition in $(i)$, we see that the condition (a) in part (1) of \cite[Lem.~3.1]{RPC} is not satisfied, with $d_1=d$ and $d_2=\frac{a^2-4b}{d}$. 

On the other hand, it is worth noting that the negation of condition (b) in part (1) of \cite[Lem.~3.1]{RPC} is equivalent to having $\nu_p(b) = \mathrm{odd}$ or $\left(\frac{a}{p}\right) = 1$, both of which are the assumptions that we have already made.
\end{proof}
Now the idea is that we count the number of $(a,b)$ satisfying the property: there exists at least one prime $q \mid s(a^2-b)$, for which there is a prime $p$ with $\left(\frac{q}{p}\right)=-1$ and $\left(\frac{a}{p}\right)=1$. For this, it is enough to eliminate the pairs $(a,b)$ with the property: given any pair of signs $(\sigma_1,\sigma_2)$, for any prime $q \mid s(a^2-b)$, and for any prime $p$ satisfying $\left(\frac{a}{p}\right)=\sigma_2$, we have $\left(\frac{q}{p}\right)=\sigma_1$. In particular, our job will be done by taking any pairs of signs $(\sigma_1,\sigma_2)\neq (-1,1)$.

\begin{lemma}\label{lem:-mbound}
Let $\sigma_2\in \{\pm 1\}$ be any sign, $a\in \Z$ be an integer which is not a square, and $p$ be a prime with $\left(\frac{a}{p}\right)=\sigma_2$. Then, we have the following estimate, as we take $X\to \infty$
$$N_a(X)=\#\Big\{b\in \Z:|b|\leq X,~q~\mathrm{prime}\mid s(a^2-b)\implies \left(\frac{q}{p}\right)=\sigma_1\Big\}=O\left((X+a^2)^{1/2}+\frac{X+a^2}{(\log (X+a^2))^{1/2}}\right),$$
where the implicit constant is absolute. In particular, we have the estimate as $X\to \infty$
\begin{equation}\label{eqn:mainestimate}
\Big\{E_{a,b}: H_2(E_{a,b})\leq X,~q~\mathrm{prime}\mid s(a^2-b)\implies \left(\frac{q}{p}\right)=\sigma_1\Big\}=O\left(\frac{X^{3}}{\log X}\right),
\end{equation}
where the implicit constant is absolute.
\end{lemma}
\begin{proof}
  Given any $a$ and $p$ as per the requirement, the set of primes $q$ with $\left(\frac{q}{p}\right)=1$, has density $1/2$. In other words, we need to demonstrate that for a rare integer $b$, all the prime factors of $s(a^2-b)$ belong to a set of primes with density $1/2$. For any parameter $y$, the number of integers $0<n\leq y$ for which all the prime factors of $n$ belong to a set of primes with density $1/2$ grows asymptotically like $\frac{y}{\sqrt{\log y}}$. This observation shows that
 \begin{align}\label{eqn:estimate}
N_a(X)&\leq \sum_{n\leq (X+a^2)^{1/2}}\frac{1}{n^2}\frac{X+a^2}{(\log (X+a^2)-\log n^2)^{1/2}}\nonumber\\
&=\sum_{(X+a^2)^{1/4}<n\leq (X+a^2)^{1/2}}\frac{1}{n^2}\frac{X+a^2}{(\log (X+a^2)-\log n^2)^{1/2}}+\sum_{n\leq (X+a^2)^{1/4}}\frac{1}{n^2}\frac{X+a^2}{(\log (X+a^2)-\log n^2)^{1/2}}.
\end{align}
It is evident that the first sum in (\ref{eqn:estimate}) is $O((x+a^2)^{1/2})$. Moreover, the second sum in (\ref{eqn:estimate}) is $O\left(\frac{X+a^2}{(\log (X+a^2))^{1/2}}\right)$, and this gives the desired estimate for the quantity $N_a(X)$. 

Let us now recall from Section~\ref{sec:E2setup} that $H_2(E_{a,b})\leq X$ implies $|a|\leq X$ and $|b|\leq X^2$. The required estimate at (\ref{eqn:mainestimate}) then follows directly by estimating the sum $\sum\limits_{a\leq X}N_a(X^2)$.
  \end{proof}
Consequently, we obtain the following implication on the $\phi$-Selmer ranks.
\begin{cor}\label{cor:main}
    For a positive proportion of elliptic curves in $\E_2$, we have 
$\mathrm{rank}(E(\Q))\leq \omega(N(E))-2$.

\end{cor}

\begin{proof}
 Lemma~\ref{lem:-mbound} implies that, for almost all pairs $(a,b)\in \Z^2$, there exist primes $p,q$ (dependent on $a,b$) such that $\nu_q(a^2-4b)=\mathrm{odd}$, and the divisor $d=q^{\nu_q(a^2-4b)}$ satisfies $\left(\frac{d}{p}\right)=-1$. Specifically, Lemma~\ref{lem:phisriteria} implies $C_{d,\frac{a^2-4b}{d}}(\Q_p)=\phi$. Hence, the proof follows combining Lemma~\ref{lem:-mbound} and Corollary~\ref{cor:rankmain}.
\end{proof}
To prove the main result of this section, we need one more result.
\begin{lemma}\label{lem:z2}
    The $2$-torsion group for $100\%$ of the all elliptic curves in $\E_2$ is $\Z/2\Z$.
\end{lemma}

\begin{proof}
    It is enough to establish that, for $100\%$ of all pairs $(a,b)\in \Z^2$, the corresponding quadratic equation $x^2+ax+b=0$ has no rational solutions. In particular, this means that the discriminant $a^2-4b$ is not a perfect square. This assertion is evident since, for any fixed square $S$, there are at most $O(X)$ pairs $(a,b)\in \Z^2$ satisfying $|a|\leq X$, $|b|\leq X^2$, and $a^2-4b=S$. However, since there are only $O(X)$ choices for $S$, this completes the argument.
\end{proof}

\subsection*{Proof of Theorem~\ref{thm:uncond}}
Lemma~\ref{lem:z2} and Corollary~\ref{cor:E2} imply that, for almost all elliptic curves $E$ in $\E_2$, we have $\omega(N(E))-2\leq \nu_2(m_E)$. Therefore, the result can be deduced from Corollary~\ref{cor:main}.
\qed

\subsection{$M$-Watkins over $\mathcal{E}_2$} For any integer $M$, and for each integer $a$ which is not a perfect square, we consider the following set of primes $\mathbb{P}_a=\{p~\mathrm{prime}: \left( \frac{a}{p}\right)=1\}$. We pick the least prime $p>3$ from $\mathbb{P}_a$ and denote it as $P(a)$. For any integer $i$, we also denote $q_i$ to be the $i$th least prime for which $\left( \frac{q_i}{p}\right)=-1$, and set $d_i=q_i,~\forall 1\leq i\leq M$. We then consider $N_{a}=\Big\{b\in \Z: \nu_{q_i}(a^2-4b)=1,~\forall 1\leq i\leq M\Big\}.$
\begin{lemma}\label{lem:mbound}
  For any integer $a$ that is not a square, and and for any $b\in N_{M}(a)$ which is co-prime to $a$, we have the following rank bound $$\mathrm{rank}(E_{a,b}(\Q))\leq \omega(N(E_{a,b}))-M.$$  
\end{lemma}
\begin{proof}
Lemma~\ref{lem:phisriteria} implies that the $\phi$-Selmer rank in (\ref{eqn:1sphi}) is at most $\omega(a^2-4b)-M$. Furthermore, the $\hat{\phi}$-Selmer rank in (\ref{eqn:1sphi'}) of course at most $\omega(b)$. Since we assume that $a,b$ are co-prime, $\omega(b(a^2-4b))\leq \omega(N(E_{a,b}))+1$. The result now follows from the rank bound at (\ref{eqn:2rankbound}).
\end{proof} 
Now, we set the quantity $d_1d_2\cdots d_M$ to be $D_M(a)$, and then we have the following estimation of $D_M(a)$.
\begin{lemma}
 For any integer $a\equiv 2 \pmod 5$ that is not a square, we have $D_M(A)\asymp M^M$.   
\end{lemma}
\begin{proof}
    For any $a\equiv 2 \pmod 5$, we have $P(a)=5$. Now we choose $q_i$ as the $i$th least prime for which $q_i\equiv 2 \pmod 5$. This shows that,
    $D_M(a)\asymp_{M} \prod\limits_{1\leq i\leq M} i\log i\sim M^M$, as desired.
\end{proof}
Now, for the counting purposes, we denote 
$$N_{M}(X)=\#\{a,b\in \Z : H_2(E_{a,b})\leq X,~\mathrm{rank}(E_{a,b}(\Q))\leq \omega(a^2-4b)-M\}.$$
\begin{lemma}\label{cor:Sboundprop}
For any integer $M$, we have $N_M(X)\geq \frac{4}{5}\frac{1}{(2M)^M}X^3$, as $X\to \infty$.
   \end{lemma}
\begin{proof}
Lemma~\ref{lem:mbound} implies that,
\begin{align*}
    N_M(X)&\geq  \#\Big\{a,b\in \Z: (a,b)=1,~|a|\leq X,~|b|\leq X^2,~\nu_{q_i}(a^2-4b)=1,~\forall 1\leq i\leq M\Big\}\\
    & =\sum_{|a|\leq X} \#\Big\{b\in \Z: (b,a)=1,~|b|\leq X^2,~q_i\mid a^2-4b,~q_i^2\nmid a^2-4b,~\forall 1\leq i\leq M\Big\}\\
&= 2X^2\left(\sum_{|a|\leq X} \frac{\phi(a)}{a}\prod_{i=1}^{M}\left(\frac{1}{q_i}-\frac{1}{q_i^2}\right)\right)+O_{M}(X)\\
&\geq 2X^2\left(\sum_{|a|\leq X} \frac{\phi(a)}{a}\prod_{i=1}^{M}\frac{1}{2q_i}\right)+O_{M}(X)\\
&\geq 2X^2\left(\sum_{\substack{|a|\leq X\\ a\equiv 2\pmod 5}}\frac{\phi(a)}{a}\frac{1}{2^MD_M(a)}\right)+O_{M}(X)\\
&= \frac{1}{(2M)^M}\frac{1}{\zeta(2)}\frac{4}{5}X^3+O_{M}(X^2\log X),
\end{align*}
and this completes the proof, as desired.
\end{proof}
\subsection{Proof of Theorem~\ref{thm:e2mwatkins}}
The proof follows combing Lemma~\ref{lem:mbound} and Lemma~\ref{cor:Sboundprop}.\qed

\subsection{Low $2$-Selmer rank for a specific sub-family of $\mathcal{E}_2$}\label{sec:low2}
In the introduction, we highlighted that Watkins's conjecture is verified for elliptic curves having a rank no greater than $1$. This serves as our inspiration to investigate pairs $(A,B)$ where the corresponding elliptic curve $E_{A,B}$ has a rank of either $0$ or $1$. This section of our study heavily relies on exploring quadratic twists of a specific elliptic curve. Given an elliptic curve $E_{A,B}/\Q$, we delve into the collection of quadratic twists denoted as ${E_{A,B}^{(D)}}$. Lozano and Pasten showed in \cite{EP21} that Watkins's conjecture is true for any $D$ with 
\begin{equation}\label{eqn:omegareq}
\omega(D)\geq 6+5\omega(N(E))-\nu_2(m_E/c^2_E).
\end{equation}
When the elliptic curve $E_{A,B}$ has prime power conductor, Caro in \cite{caro22watkinss}, using Setzer's classification \cite{Setzer75}, established the validity of Watkins's conjecture for any $E_{A,B}^{(D)}$. 

In this section, we specifically focus on the elliptic curve denoted as $E_{0}:=y^2=x^3-1$. Notably, we have the values $N(E_0)=144$, $m_{E_0}=64$, and $E_0[2](\Q)=\Z/2\Z$. Therefore, we are not able to apply Caro's result for $E_0$. However, using equation (\ref{eqn:omegareq}), it turns out that Watkins's conjecture is valid for $E_0^{(D)}$ when $\omega(D)\geq 10+2\nu_2(c_{E_0})$. To address the remaining cases, we rely on the crucial results derived using techniques outlined in \cite{JMR}.

\begin{lemma}\label{lem:rank0} 
Let $D$ be any square-free integer with one of the following properties:
\begin{enumerate}
    \item[(i)] all prime factors are $5 \pmod{12}$.
    \item[(ii)] all but exactly one prime factors are $5 \pmod {12}$, and the leftover prime is $-1 \pmod 4$.
\end{enumerate}
Then for any such $D$, we have $\mathrm{rank}(E_0^{(D)}(\Q))\leq 1$.

\end{lemma}
\begin{proof}
If $D$ satisfies the condition at $(i)$, then according to \cite[Prop.~1]{JMR}, we have 
$\mathrm{dim}_{\mathbb{F}_2}(\mathrm{Sel}^2(E_0^{(D)}))=\mathrm{dim}_{\mathbb{F}_2}(\mathrm{Sel}^2(E'^{(D)}))=1$, and this shows that $\mathrm{rank}(E^{(D)}(\Q))= 0$. Now, suppose $D$ satisfies the condition at $(ii)$. Applying case (4) of \cite[Lem.~4]{JMR} and case (5) of \cite[Lem.~5]{JMR}, we obtain
$\mathrm{dim}_{\mathbb{F}_2}(\mathrm{Sel}^2(E^{(D)}))+\mathrm{dim}_{\mathbb{F}_2}(\mathrm{Sel}^2(E'^{(D)}))\leq 3,$
and this shows that $\mathrm{rank}(E^{(D)}(\Q))\leq 1$.
\end{proof}

\begin{remark}\rm
We can establish that $\mathrm{rank}(E^{(p)}(\Q))\leq 1$ for any prime $p$ under the conditions that $p\equiv 5 \text{ or } 11 \pmod {12}$. Specifically, this result is stated as \cite[Thm.~2]{JMR} for $p\equiv 5 \pmod {12}$, while for $p\equiv 11 \pmod {12}$, it can be deduced by applying the proof of \cite[Thm.~2]{JMR}.
\end{remark}
As a result of the aforementioned analysis, we can derive the following extension of \cite{EP21} regarding Watkins's conjecture.
\begin{cor}\label{cor:another}
For any squarre-free integer $D$, we have $E_0^{(D)}\in \mathcal{E}_2$. Then, Watkins's conjecture holds true for $E_0^{(D)}$, provided that $D$ satisfies one of the following:
\begin{enumerate}
    \item[(i)] $\omega(D)\geq 10+2\nu_2(c_{E_0})$
\item[(ii)] $\omega(D)< 10+2\nu_2(c_{E_0})$, and $D$ satisfies one of the conditions specified in Lemma~\ref{lem:rank0}.
\end{enumerate}
\end{cor} 
\begin{proof}
It is evident that $E_0^{(D)}[2](\Q)=\Z/2\Z$ for any $D$. Now the proof for part $(i)$ follows from \cite{EP21}, while the proof of part $(ii)$ uses Lemma~\ref{lem:rank0} and the fact that Watkins's conjecture is true for elliptic curves of rank at most $1$.
\end{proof}
\subsection*{Proof of Theorem~\ref{thm:x3+k}}
\subsubsection*{Proof of part (iii)} 
The size of $\mathcal{E}_2(X)$ is computed in (\ref{eqn:E2}) in Section~\ref{sec:E2setup}. For the $M$-Watkins elliptic curves, we shall use Lemma~\ref{lem:rank0}. Note that $H(E_0^{(D)})\leq X$, if and only if, $|D|\leq X^{1/2}$. Due to Lemma~\ref{lem:rank0}, it is enough to estimate 
$$\Big\{D\in \Z: |D|\leq X^{1/2},~\omega(D)>M,~p~\mathrm{prime}\mid D \implies p\equiv 5 \pmod{12}\Big\}.$$
Let us note that
\begin{equation}\label{eqn:lessm}
\#\Big\{D\in \Z: |D|\leq X^{1/2},~\omega(D)\leq M\Big\}=O_{M}\left(\frac{X^{1/2}}{\log X}(\log \log X)^M\right).
\end{equation}
Moreover, since $\{p~\mathrm{prime}:p\equiv 5 \pmod{12}\}$ has density $1/4$, we also have
\begin{equation}\label{eqn:14}
\#\Big\{D\in \Z: |D|\leq X^{1/2},~~p~\mathrm{prime}\mid D \implies p\equiv 5 \pmod{12}\Big\}=\frac{X^{1/2}}{(\log X)^{1/4}}.
\end{equation}
The desired estimate now follows combining (\ref{eqn:lessm}) and (\ref{eqn:14}).
\qed

To prove part $(ii)$ of Theorem~\ref{thm:3-torsion}, we first need to study the Watkins for the elliptic curves of \textit{Type-I}. Consequently, we will deduce part $(ii)$ along with proving Theorem~\ref{thm:x3+k}.
\subsection{Watkins over $\mathcal{E}_3$}\label{sec:we3}
Let us begin by revisiting the discussion and notations outlined in Section~\ref{sec:3descent}. For any integer $a$, consider the elliptic curve $E_a:=y^2=x^3+a$. In this section, we shall study Watkins's conjecture for the family of elliptic curves $\{E_a\}$. 
\subsection*{Proof of Theorem~\ref{thm:x3+k}}
\begin{proof}[Proof of part (i)]
It is enough to show that $S=\#\Big\{a\in \mathbb{Z}: |a|\leq X,~\mathrm{rank}(E_a)> \omega(a)-M\Big\}$ is $O_M\left( \frac{X}{\log \log X}\right)$. To do this, let us first note that \cite[Thm.~3]{3-iso} implies 
\begin{equation*}\label{eqn:what}
\sum_{\substack{a\in S}} (\omega(a)-M)\leq \sum_{\substack{a\in S }} \mathrm{rank}(E_a)=O(X).
\end{equation*}
Since $S \leq 2X$, we also have that $\sum\limits_{\substack{a\in S}} \omega(a)=O(X)$. On the other hand, it follows from \cite[Thm.~3.2.3]{Murty-Alina} that $\sum\limits_{\substack{a\in S}} \Big|\omega(a)-\log \log X\Big|^2=O\left(X\log \log X\right)$. In particular, applying triangle inequality, we have 
$\#~S~(\log \log X)^2=O\left(X\log \log X\right)$, and this completes the proof.
\end{proof}
With this, we are now prepared to investigate Watkins's conjecture over $\E_3$. Combining (\ref{eqn:3rankbound}), (\ref{eqn:sphi}), and (\ref{eqn:sphi'}), we have
\begin{equation}\label{eqn:rkbound}
\rank\left(E_a(\Q)\right) \leq  r_3(K_a)+\dim_{\F_3}\left( U_{K_a}/U_{K_a}^3\right)+ r_3(K_{-27a})+\dim_{\F_3}\left( U_{K_{-27a}}/U_{K_{-27a}}^3\right)+\#~S_a+\#~S_{-27a}.
\end{equation}

Now, let us recall the set of primes $S_a$, defined by $2,3\in S_a$, and $p\neq 2,3 \in S_a \Longleftrightarrow \nu_p(a)=2,4~ \text{and} ~ \left(\frac{-3}{p}\right)=1$. Our objective is to establish an upper bound on $\#~S_{a}+\#~S_{-27a}$ over a suitable set of integers $a$.  By definition, we know that $\#~S_{a}=\#~S_{-27a}$. Hence, our goal is to determine the proportion of all $a$ for which $\#~S_{a}$ is significantly smaller than $\frac{\omega(a)}{2}$. 
\subsubsection{Bounding $r_3$ part}
For any integer $a$, let us recall that $K_a=\Q(\sqrt{-3a})$, and $r_3(K_a)$ denotes order of the $3$-torsion part of the class group of $K_a$. Note that $r_3(K_a)=O(1)$, when $a$ is a square, up to a constant factor. Based on this observation, we can derive the following result.
\begin{lemma}\label{lem:sqcase}
For any integers $m,k_1,k_2,\cdots k_m$, consider the set of integers
$N_{\vec{k},m}=\bigcup \limits_{i=1}^{m}\{k_in^2:n\in \Z\}$, where $\vec{k}=(k_1,k_2,\cdots,k_m)$. For any $a\in N_{\vec{k},m}$, we have 
\begin{equation}
    r_3(K_a)+\dim_{\F_3}\left( U_{K_a}/U_{K_a}^3\right)+ r_3(K_{-27a})+\dim_{\F_3}\left( U_{K_{-27a}}/U_{K_{-27a}}^3\right)=O_{\vec{k}}(1).
\end{equation}
\end{lemma}

Let us now note that, the discussions in Section~\ref{sec:E2setup} and Section~\ref{sec:E3setup} implies that
\begin{align}\label{eqn:Eatype}
    E_a \in \begin{cases}
       \mathcal{E}_2 & \text{if $a$ is a cube,}\\
         \mathcal{E}_3, & \text{if $a$ is a square.}\\
    \end{cases}
\end{align}

\subsubsection*{Proof of part \textit{(ii)}} First of all for any $a\in N_{\vec{k},m}$, we have
\begin{equation}\label{eqn:Eabound}
\mathrm{rank}(E_a(\Q))\leq 2~\# S_{a}+O_{\mathcal{S}}(1).
\end{equation}
For any integer $M>0$, it is now enough to show that $\#S_a\leq \frac{\omega(a)}{2}-M$, for at least $\gg_{\vec{k},m}\left(\frac{X}{\log X}\right)^{1/2}$ many elements in $N_{\vec{k},m}(X)$. Since $N_{\vec{k},m}(X)$ is contained in $\bigcup\limits_{i=1}^{m}\{k_in^2\}$, it is enough to only consider $a$ of the form $kn^2$, for a given integer $k$.

We have $\#\Big\{a\in \Z: |a|\leq y^{1/2}:~p\mid a\implies \left(\frac{-3}{p}\right)=-1\Big\}\gg \left(\frac{y}{\log y}\right)^{1/2}$, as the limit $y\to \infty$. On the other hand for any integer $M$, we also have $\#\Big\{a\in \Z: a\leq y,\omega(a)\leq M\Big\}=O_{M}\left(\frac{y^{1/2}}{\log y}(\log \log y)^M\right)$.

Since we have, $\left(\frac{y}{\log y}\right)^{1/2}\gg \frac{y^{1/2}}{\log y}(\log \log y)^M$, the proof is complete taking $y=\min\left\{\frac{X^{1/2}}{k_i}\right\}_{1\leq i\leq m}$.
\qed
\subsection*{Proof of Corollary~\ref{thm:3-torsion}} 
First of all, the size of $\mathcal{E}_3(X)$ follows from (\ref{eqn:E3}) in Section~\ref{sec:E3setup}. The lower bound on the number of $M$-Watkins elliptic curves in $\mathcal{E}_3$, taking $m=1$ and $k_1=1$ in part $(i)$ of Theorem~\ref{thm:x3+k}.
\qed
\begin{remark}\rm
We expect that $\# S_a\leq \frac{\omega(a)}{2}-M$ will hold for $50\%$ of the times over $N_{\vec{k},m}$. This is because for any set of primes $S$ of density $1/2$ if we denote $\omega_S(n)$ be the number of prime factors of $n$ from $S$, then even though $\omega_S(n)-\frac{\omega(n)}{2}$ typically is small, it is negative if and only if $\omega_{S'}(n)-\frac{\omega(n)}{2}$ is positive. Here $S'$ is the complement of $S$, which also has a density of $1/2$. Since $S$ and $S'$ both are of density $1/2$, we see no reason to believe that signs in each of the sets $\{\omega_{S}(n)-\frac{\omega(n)}{2}\}_{n\in \Z}$ and $\{\omega_{S'}(n)-\frac{\omega(n)}{2}\}_{n\in \Z}$ are not equidistributed.
\end{remark}
\section{$M$-Watkins over $\E_3,\E_5$ and $\E_7$}\label{sec:higher}
The key tool that we will be using is the main counting lemma due to Harron, Snowden~\cite{HS17}, designed to count the twists of a given $1$-parameter family of elliptic curves.
\begin{lemma}\label{lem:veryimp}
Let us consider $f_{\ell}$ and $g_{\ell}$, for $\ell \in \{3,5,7\}$. Denote $\mathcal{E}_{\ell}(X)$ be the set of pairs $(A,B) \in \Z^2$ satisfying the following conditions:
\begin{itemize}
\item[(i)] $4A^3+27B^2 \ne 0$.
\item[(ii)] $\gcd(A^3,B^2)$ is not divisible by any 12th power.
\item[(iii)] $|A|<X^2$ and $|B|<X^{3}$.
\item[(iv)] There exist $u,r \in \Q$ such that $A=u^4 f(r)$ and $B=u^6 g(r)$.
\end{itemize}
Denote $\mathcal{E}'_{\ell}(X)$ to be the set of all $(u,r)\in \Q^2$ appearing in the last point above. Then, we have
\begin{equation}\label{eqn:contain}
\Big\{r:(u,r)\in \mathcal{E}'_{\ell}(X)\Big\}\subseteq \Big\{a/b: |a|=O_{f_{\ell},~g_{\ell}}\left(|X|^{m_{\ell}}\right),~|b|=O_{f_{\ell},~g_{\ell}}\left( X^{n_{\ell}}\right)\Big\},
\end{equation}
where the exponents $m_{\ell},n_{\ell}$ are explicitly given by
\begin{align*}
(m_{\ell},~n_{\ell})=\begin{cases}
    (3/2,~1/2), & \text{if $f=f_5,~\mathrm{and}~g=g_5$},\\
    (1/4,~1/4), & \text{if $f=f_5,~\mathrm{and}~g=g_5$},\\
           (1/2,~1/2), & \text{if $f=f_7,\mathrm{and}~g=g_7$.}\\
\end{cases}
\end{align*}
\end{lemma}

\begin{proof}
    It follows from \cite[Table 2]{HS17} that the relevant $m$ is $1$ for both $f_5$ and $f_7$. In particular, \cite[Lem.~2.5]{HS17} implies the result for $\ell=5$ and $7$. Now for $\ell=3$, \cite[Table 2]{HS17} and \cite[Lem.~2.4]{HS17} indicates that $r$ can be written as $a/b^2$ such that $\mathrm{gcd}(a,b^2)$ is square-free. Consequently, the result follows from \cite[Lem.~2.5]{HS17}.
\end{proof}
Consequently, we shall deduce a result on the average rank for the fibers of an elliptic surface over $\Q(t)$, which will play a key role in this section. As mentioned in the introduction, it is recently conjectured by Park, Poonen, Voight, and Wood in \cite{rank21} that there are only finitely many elliptic curves of greater than rank $21$. For this section, we only need to know that only a small proportion of the elliptic curves have a large rank. In this regard, we prove the following.

\begin{lemma}\label{lem:rankbound}
Consider $f_{\ell}$ and $g_{\ell}$, for $\ell \in \{3,5,7\}$, and assume that the conjectures $B, C$ and $D$ in \cite{Silverman} hold. Then for any integer $M\geq 1$, the proportion of elliptic curves $E \in \mathcal{E}_{f,g}$ with a rank greater equals to $M$ is at most $O_{f,g}\left(\frac{1}{M}\right)$.
\end{lemma}

Additionally, we now need to examine the conductors of the elliptic curves in $\mathcal{E}_{\ell}(X)$. Proposition~\ref{prop:maintool} is central to this analysis. For the preparations of the two tasks as described above, we require the following facts for polynomials $f_3$, $f_5$, $f_7$ (and their counterparts $g_3$, $g_5$, $g_7$):

\begin{lemma}\label{lem:irreducible}
Let us consider the polynomials $\Delta_{\ell} = 4f_{\ell}^3 + 27g_{\ell}^2$, for $\ell \in \{3, 5, 7\}$. It follows that any of these polynomials, namely $\Delta_3$, $\Delta_5$, and $\Delta_7$, have at least one non-trivial irreducible factor in $\mathbb{Z}[t]$ that divides with multiplicity one.
\end{lemma}

To prove both Lemma~\ref{lem:irreducible} and Lemma~\ref{lem:rankbound}, we explicitly write down the equations of $\Delta_{\ell}$ for $\ell=3,5$ and $7$, which we briefly do in the next section.
\subsection{Structures of $f_3,f_5$ and $f_7$}
\subsubsection*{Equation for $\Delta_3$} 
The authors in \cite[Lem.~3.4]{HS17} demonstrate that the elements in $\mathcal{E}_3$ can be represented by the equation $y^2+axy+by=x^3$, where $P=(0,0)$ is a point on any such curve with order $3$. Notably, when $a=0$, this corresponds to a family that covers only $0\%$ of the elements in $\mathcal{E}_3$. This aspect is crucial because the curve $y^2+by=x^3$ is isomorphic to $y^2=x^3+16b^2$.

As a consequence, \cite[Prop.~3.5]{HS17} establishes that nearly all elements in $\mathcal{E}_3$ can be expressed through integral polynomials as follows: 
$$f_3(t)=3^4\left(2t-\frac{1}{3}\right),~~~g_3(t)=3^6\left(t^2 +\frac{2}
{3}t +\frac{2}
{27}\right).$$
In this case, we have 
\begin{equation}\label{eqn:delta3}
    \Delta_3(t)=4f_3(t)^3-27g_3(t)^2=-14348907 t^4-2125764 t^3- 17006112 t^2-157464,
\end{equation}
which is irreducible in $\Z[t]$.

Let us recall that the Tate normal form of an elliptic curve with $P = (0, 0)$ is given by $E = E(b, c): Y^2 + (1-c)XY-bY = X^3-bX^2.$ The discriminant of $E(b, c)$ is $\Delta(b, c) = b^3(16b^2-8bc^2-20bc+b+c(c-1)^3)$. With this, one has the following:
\begin{equation}
5P=\infty \iff b=c,~7P=\infty \iff b = t^3 -t^2,~c = t^2-t.
\end{equation}

\subsubsection*{Equation for $\Delta_5$} 

The corresponding family of elliptic curves is 
$\left\{Y^2+(1-t)XY-tY=X^3-tX^2\right\}_{t \in \Q}$. A simple computation of SageMath, immediately gives us the discriminant 
\begin{equation}\label{eqn:delta5}
\Delta_5(t)=(t^2 - 11t - 1)t^5.
\end{equation}

\subsubsection*{Equation for $\Delta_7$} The corresponding family is
$\left\{Y^2 + (1-t^2+t)XY-(t^3-t^2)Y = X^3-(t^3-t^2)X^2\right\}_{t\in\Q}$. Again, a SageMath computation gives us
\begin{equation}\label{eqn:delta7}
    \Delta_7(t)=(t^3 - 8t^2 + 5t + 1)(t-1)^7t^7.
\end{equation}

\subsection*{Proof of Lemma~\ref{lem:irreducible}}
The proof is based on the explicit equations derived in the earlier sections for $\Delta_3$, $\Delta_5$, and $\Delta_7$, respectively. More specifically, the proof follows from the equations (\ref{eqn:delta3}), (\ref{eqn:delta5}) and (\ref{eqn:delta7}).
\qed

\subsection{Aerage rank of elliptic curves and the conclusions}\label{sec:avgrank} 
Let $\mathcal{E}/\mathbb{Q}(t)$ be an elliptic surface, i.e., a $1$-parameter family of elliptic curves, given by $$\mathcal{E}_{f,g}:y^2=x^3+f(t)x+g(t),~f(t),g(t)\in \Z[t].$$ For each rational $r\in \Q$, let us consider the curve $E_{f(r),g(r)}:= y^2=x^3+f(r)x+g(r)$. The main goal of this section is to study the ranks of the elliptic curves $E_{f(r),g(r)}$, as we vary $r$ over all the rationals. We follow the approach in \cite{Silverman}, and obtain the following estimate on average rank.

\begin{prop}\label{prop:avgrank}
Let $f,g\in \Z[t]$ be two arbitrary polynomials, and assume that the conjectures $B, C$, and $D$ in \cite{Silverman} hold. For any two parameters $X,Y\to \infty$, we have the following estimate
    $$ R_{f,g}(X,Y)=\sum\limits_{\substack{r=a/b\in \Q\\|a|\leq X,~|b|\leq Y}}\mathrm{rank}(E_{f(r),g(r)})=O_{f,g}(XY),$$
    provided that $\mathcal{E}_{f,g}$ is a non-split elliptic surface.
\end{prop}
Before proving Proposition~\ref{prop:avgrank}, let us first consider the average of Frobenius trace over the fibers of any elliptic surface $\mathcal{E}_{f,g}/\Q(t),~f,g\in \Z[t]$, and any prime $p$, as follows
$$A_p(\mathcal{E})=\frac{1}{p}\sum_{\bar{r}\in \mathbb{F}_p}a_p(\mathcal{E}_{\bar{r}}),$$
where we denote $a_p(E)$ to be the the quantity $p+1-\#E(\mathbb{F}_p)$, for any elliptic curve $E/\Q$.
\begin{lemma}\label{lem:uniform}
For any polynomials $f,g\in \Z[t]$, and for any prime $p$, we have the estimate
    $$A_p(\mathcal{E})\leq 3\deg(\Delta(\mathcal{E}_{f,g}))+\deg(c_4(\mathcal{E}_{f,g}))-2.$$
    In particular, $A_p(\mathcal{E})$ is bounded independent of $p$ and $\mathcal{E}$, as the quantity on the right above is uniformly bounded.
\end{lemma} 
\begin{proof}
It follows from \cite[Prop.~1.2]{Mich} that $$\sum\limits_{\substack{\bar{r}\in \mathbb{F}_p\\\Delta(\bar{r})\neq 0}}a_p(\mathcal{E}_{\bar{r}})\leq p(2\deg(\Delta(\mathcal{E}_{f,g}))+\deg(c_4(\mathcal{E}_{f,g}))-2).$$
The proof follows immediately, as the number of $\bar{r}\in \F_p$ with $\Delta(\bar{r})=0$ is at most $\deg(\Delta)$.
\end{proof}
\subsection*{Proof of Proposition~\ref{prop:avgrank}}
Let us denote $d=\max\left\{\mathrm{deg}(f),\mathrm{deg}(g)\right\}$. For any rational $r=a/b$, note that $E_{f(r),g(r)}$ is isomorphic to $E_{b^{4d}f(r),b^{6d}g(r)}$ over $\Q$. Additionally, both $b^{4d}f(r)$ and $b^{6d}g(r)$ are integers. For each $b$, observe that both of $f_b(t):=b^{4d}f(t/b)$, and $g_b(t):=b^{6d}g(t/b)$, are polynomials with integer coefficients. 

For each $b\in \Z$, let us consider the elliptic surface
$$\mathcal{E}_{f_b,g_b}:y^2=x^3+f_b(t)x+g_b(t),~f_b(t),~g_b(t)\in \Z[t].$$
It is clear that $\mathrm{rank}(\mathcal{E}_{f,g}(\Q(t))=\mathrm{rank}(\mathcal{E}_{f_b,g_b}(\Q(t))$, for any $b\in \Z$. Additionally, $\Delta(E_{f,g})(\alpha)=0$ if and only if, $\Delta_{E_{f_b,g_b}}(b\alpha)=0$. In particular, $\mathrm{deg}(\Delta_{\mathcal{E}_{f,g}})=\mathrm{deg}(\Delta_{\mathcal{E}_{f_b,g_b}})$. Similarly, we also have $\mathrm{deg}(c_4(\mathcal{E}_{f,g}))=\mathrm{deg}(c_4(\mathcal{E}_{f_b,g_b}))$ and $\mathrm{deg}(c_6(\mathcal{E}_{f,g}))=\mathrm{deg}(c_6(\mathcal{E}_{f_b,g_b}))$. Summarizing the discussion, we have the following two crucial ingredients
\begin{equation}\label{eqn:randn}
\mathrm{rank}(\mathcal{E}_{f,g}(\Q(t))=\mathrm{rank}(\mathcal{E}_{f_b,g_b}(\Q(t)),~\mathrm{deg}(N(\mathcal{E}_{f,g}))=\mathrm{deg}(N(\mathcal{E}_{f_b,g_b})),~\forall b\in \Z.
\end{equation}
Furthermore, we can write
\begin{align}\label{eqn:sumequal}
  R_{f,g}(X,Y)=\sum\limits_{\substack{r=a/b\in \Q\\|a|\leq X,~|b|\leq Y}}\mathrm{rank}(E_{f(r),g(r)})&=\sum_{\substack{b\in \Z\\ |b|\leq Y}}~\sum_{\substack{a\in \Z\\ |a|\leq X}}\mathrm{rank}(E_{f_b(a),~g_b(a)}).
\end{align}
For any $\lambda>0$, let us now consider the triangle function $F_{\lambda}(x)=\max\{0,1-|x/\lambda|\}$. Then, we have
\begin{equation*}
 \lambda\mathrm{rank}(E_{f(r),g(r)})\leq \log~N(E_{f(r),g(r)})-2\sum\limits_{p^m\leq \lambda}F_{\lambda}(m\log p)\frac{a_{p^m}(E_{f(r),g(r)})\log p}{p^m}+O(1),~\forall r\in\Q,
\end{equation*}
where the constant term $O(1)$ is absolute, i.e., independent to $r\in \Q$. In particular, we have
\begin{align}\label{eqn:avglast}
    \lambda R_{f,g}(X,Y)&\leq \sum_{\substack{b\in \Z\\ |b|\leq Y}}~\sum_{\substack{a\in \Z\\ |a|\leq X}} \log~N(E_{f_b(a),g_b(a)})-2\sum_{m\geq 1} S_m(X,Y)+O(XY),\nonumber \\
    &\leq \sum_{\substack{b\in \Z\\ |b|\leq Y}}~\big(\sum_{\substack{a\in \Z\\ |a|\leq X}} \deg(N(\mathcal{E}_{f_b,g_b})) \log |a|+O_{f,g}(1)\big)-2\sum_{m\geq 1} S_m(X,Y)+O(XY)\nonumber\\
    &\leq \mathrm{deg}(N(\mathcal{E}_{f,~g}))XY \log X-2\sum_{m\geq 1} S_m(X,Y)+O(XY).
\end{align}
where the implicit constant is absolute, and 
$S_{m}(X,Y)=\sum\limits_{\substack{r=a/b\in \Q\\|a|\leq X,~|b|\leq Y}}\sum\limits_{p^m\leq e^{\lambda}}F_{\lambda}(m\log p)\frac{a_{p^m}(E_{f(r),g(r)})\log p}{p^m}.$

Similarly as in \cite[p.~232]{Silverman}, we have $\sum\limits_{m>2}S_m(X,Y)=O(XY)$, where the implicit constant is absolute. Now, we need to estimate $S_1(X,Y)$ and $S_2(X,Y)$. Note that, 
\begin{align}\label{eqn:s2}
S_2(X,Y)=\sum\limits_{\substack{r=a/b\in \Q\\|a|\leq X,~|b|\leq Y}}\sum\limits_{p\leq e^{\lambda/2}}F_{\lambda}(2\log p)\frac{a_{p^2}(E_{f(r),g(r)})\log p}{p^2}=-2XY\lambda+o(XY),
\end{align} 
where the last estimate follows from \cite[Lem.~4.2]{Silverman}. Now for the estimation of $S_1$, note that
\begin{align}\label{eqn:s3}
S_1(X,Y)&=\sum\limits_{|b|\leq Y}\sum\limits_{|a|\leq X}\sum\limits_{p\leq e^{\lambda}}F_{\lambda}(\log p)\frac{a_{p}(E_{f_b(a),g_b(a)})\log p}{p}\nonumber\\
&=\sum\limits_{|b|\leq Y} \sum\limits_{p\leq e^{\lambda}}F_{\lambda}(\log p)\frac{2X}{p}A_p(\mathcal{E}_b)+\sum\limits_{|b|\leq Y}\sum\limits_{p \leq e^{\lambda}}\sum_{\bar{r}\in I_p\subset \mathbb{F}_p}F_{\lambda}(\log p)\frac{\log p}{p}a_{p}(E_{f_b(\bar{r}),g_b(\bar{r})}),
\end{align}
for some subset $I_p$ of $\mathbb{F}_p$. Lemma~\ref{lem:uniform} and \cite[Cor.~3.2]{Silverman} implies that the first sum above is 
\begin{equation}\label{eqn:fsum}
-\left(X\lambda +o(X\lambda)\right)\sum\limits_{|b|\leq Y}\mathrm{rank}(\mathcal{E}_b(\Q(t)))=-XY \mathrm{rank}(\mathcal{E}(\Q(t)))(2\lambda+o(1)).
\end{equation} 
Now, the inner sum in the second sum in (\ref{eqn:s3}) is crudely bounded by $\sum\limits_{p\leq e^{\lambda}} p^{1/2}\log p=e^{3\lambda/2}$. In particular, now combining (\ref{eqn:avglast}), (\ref{eqn:s2}), and (\ref{eqn:s3}), we have
$$R_{f,g}(X,Y)\leq \frac{\mathrm{deg}(N(\mathcal{E}_{f,~g}))}{\lambda}XY\log X-XY \mathrm{rank}(\mathcal{E}(\Q(t)))(2+o(1/\lambda))+Ye^{3\lambda/2}.$$
Now the proof follows taking $\lambda=\frac{\log X}{3}$.
\qed

As an immediate consequence, we deduce that the most elliptic curves in the $1$-parameter family as in Proposition~\ref{prop:avgrank} have a small rank. More specifically, we have the following.
\subsection*{Proof of Lemma~\ref{lem:rankbound}}
Take $E\in \mathcal{E}_{\ell}(X)$, and take the corresponding $r\in \mathcal{E}'_{\ell}(X)$. We have 
$\mathrm{rank}(E)=\mathrm{rank}(E_{f_{\ell}(r),g_{\ell}(r)})$. With this observation, it follows from Lemma~\ref{lem:veryimp} that
\begin{align}
\sum_{E\in \mathcal{E}_{\ell}(X)}\mathrm{rank}(E)&\leq \sum_{\substack{r=a/b~\in~\Q\\ |a|\leq CX^{m_{\ell}},~|b|\leq CX^{n_{\ell}}}} \mathrm{rank}(E_{f_{\ell}(r),~g_{\ell}(r)})\\
&=R_{f_{\ell},g_{\ell}}(CX^{m_{\ell}},X^{n_{\ell}})=O(X^{m_{\ell}+n_{\ell}})
\end{align}
where $C$ is the maximum of the implicit constants in (\ref{eqn:contain}), and the implicit constant is absolute. The result now follows immediately, as $\# \mathcal{E}_{\ell}(X) \asymp X^{m_{\ell}+n_{\ell}}$.
\qed

\subsection{Normal number of prime factors of square-free parts of polynomial values}
Let $f(t)\in \Z[t]$ be any polynomial of degree $d$. At any rational $r\in \Q$, let us denote $s(f(r))$ to be the square-free part of $f(r)$. Poonen \cite{Poonen}, explored the prevalence of square-free values in the polynomials. Over the integers $n$, Poonen investigated scenarios where $f(n)$ is not divisible by any fixed square $m^2 \neq 1$, utilizing the ABC conjecture. However, in our context, our focus is on understanding how frequently the square-free parts of $f(r)$ possess a substantial number of prime factors. First of all, for any rational $r=\frac{a}{b}$, let us denote the height $h(r)=\max\{|a|,|b|\}$. Then we arrange all the rationals with respect to height $h$ and prove the following. 
\begin{prop}\label{prop:maintool}
Let $M\geq 1$ be any integer, and $f(t)\in \Z[t]$ be any irreducible non-constant polynomial. Then, for almost all the rationals $r\in \Q$, we have $\omega(s(f(r)))\geq M$. More precisely, we have the following estimate:
$$\#\Big\{r\in \Q: h(r)\leq X,~\omega(s(f(r)))<M\Big\}=O_M\left( \frac{X^2}{\log \log X}\right).$$
\end{prop}
Proposition~\ref{prop:maintool} will be one of our key tools to prove the main results of this article. To prove the same, we need the following two results.
\begin{lemma}\label{lem:triv}
Let $f(t)\in \mathbb{Z}[t]$ be a non-constant irreducible polynomial, and $p$ be any given prime. Then the number of solutions to the equation $f(\bar{r})=0\pmod {p^2}$ over $\bar{r}\in \mathbb{Z}/p^2\mathbb{Z}$ is bounded by $2\mathrm{deg}(f)$, for any prime $p\nmid \mathrm{gcd}(c(f),c(f'))$. Here $f'$ is the derivative of $f$, and $c(f)$ (resp. $c(f')$) denotes the $\mathrm{gcd}$ of the coefficients of $f$ (resp. $\mathrm{gcd}$ of the coefficients of $f'$).
\end{lemma}
\begin{proof}
We can express any solution $\bar{r} \in \mathbb{Z}/p^2\mathbb{Z}$ as $\bar{r}_1 + p\bar{r}_2$, where $0 \leq \bar{r}_1, \bar{r}_2 \leq p-1$. Using this representation, we have $f(\bar{r}) = f(\bar{r}_1) + p\bar{r}_2f'(\bar{r}_1) \pmod{p^2}$. Since $f(\bar{r}) = 0\pmod{p^2}$, we have $f(\bar{r}_1)= 0 \pmod{p}$, indicating that there are at most $\mathrm{deg}(f)$ possibilities for $0 \leq \bar{r}_1 \leq p-1$. 

Now to count the number of possibilities of $\bar{r}_2$, without loss of generality let us assume that $\bar{r}_2 \neq 0$. If $p$ divides $f'(\bar{r}_1)$, then $p$ divides $\mathrm{gcd}(f(\bar{r}_1), f'(\bar{r}_1))$. But since $f(t)\in \Q[t]$ is irreducible, we know that the $\mathrm{gcd}$ of $f$ and $f'$ over $\Q[t]$ is $1$, and hence, the gcd of $f$ and $f'$ over $\Z[t]$ is $\mathrm{gcd}(c(f),c(f'))$. In particular, $p\mid \mathrm{gcd}(c(f),c(f'))$. This is not possible, as per the assumption on $p$. Therefore, we may assume that $p$ does not divide $f'(\bar{r}_1)$. Now for each possible $\bar{r}_1$, setting $f(\bar{r}_1)=p\bar{t}_1$, we can write $\bar{r}_2=-\bar{t}_1f'(\bar{r}_1)^{-1}$. In particular, $\bar{r}_2$ is uniquely determined by $\bar{r}_1$, as long as $\bar{r}_2\neq 0$. This shows that the number of possible pairs $(\bar{r}_1, \bar{r}_2)$ is only $2\mathrm{deg}(f)$, for any prime $p$ not dividing $\mathrm{gcd}(c(f), c(f'))$.
\end{proof}

\begin{lemma}\label{lem:average}
Consider $S$ to be the set of all prime factors of the leading coefficient of $f$. Then, we have the following estimate: 
$$\sum_{h(r)\leq X} \Big|\omega_{S}(s(f(r)))-\log \log X\Big|^2=O_{S}\left(X^2\log \log X\right).$$
\end{lemma}
\begin{proof}
Let us start by noting that
\begin{align}\label{eqn:hmm}
    \sum_{h(r) \leq X} \omega_S(f(r))&=\sum_{h(r)\leq X} ~\sum_{\substack{p~\mathrm{prime}~\not\in~S\\\nu_p(f(r))=\mathrm{odd}}} 1+~\sum_{h(r)\leq X}~\sum_{\substack{p~\mathrm{prime}~\not\in~S\\\nu_p(f(r))\neq 0\\ \nu_p(f(r))=\mathrm{even}}} 1\nonumber\\ 
    &=\sum_{h(r)\leq X} \omega_S(s(f(r))~+\sum_{h(r)\leq X}~ \sum_{\substack{p~\mathrm{prime}~\not\in~S\\\nu_p(f(r))\neq 0\\\nu_p(f(r))=\mathrm{even}}} 1.
\end{align}
For any rational number $r$, let us denote $r_p=r \pmod p$. Then we have,
\begin{equation}\label{eqn:impequiv}
\nu_p(f(r))>0\implies r_p\in \mathbb{F}_p,~\mathrm{and}~p\not\in S \implies \nu_p(f(r))<0,~\mathrm{if~and~only~if},~r_p=\infty.
\end{equation}
This shows that
\begin{align}\label{eqn:sumwf}
    \sum_{h(r) \leq X} \omega_S(f(r))&= \sum_{h(r) \leq X}~~~\sum_{\substack{p~\mathrm{prime}~\not\in~S\\\nu_p(f(r))>0}}1+\sum_{h(r) \leq X}~\sum_{\substack{p~\mathrm{prime}~\not\in~S\\\nu_p(f(r))<0}} 1\nonumber \\ 
    &= \sum_{h(r) \leq X}~~\sum_{\substack{p~\mathrm{prime}~\not\in~S\\r_p\in \mathbb{F}_p\\\nu_p(f(r))>0}} 1+\sum_{h(r) \leq X}~\sum_{\substack{p~\mathrm{prime}~\not\in~S\\r_p=\infty}} 1.
\end{align}
It is now enough to estimate the second sum in (\ref{eqn:hmm}), and both the sums in (\ref{eqn:sumwf}). 
\subsubsection*{Case 1: estimate of the first sum in (\ref{eqn:sumwf})} For this, let us note that for any rational $r-\frac{a}{b}$, we have
\begin{equation}\label{eqn:finv}
f\left(\frac{a+p}{b}\right)=f\left(\frac{a}{b}\right)\pmod p,~f\left(\frac{a}{b+p}\right)=f\left(\frac{a}{b}\right)\pmod p.
\end{equation}
Again, since we are assuming that $r_p\in \mathbb{F}_p$, we can write the following combining with (\ref{eqn:finv}):
\begin{align*}
    \sum_{h(r) \leq X} \omega_{S}(f(r))=\sum_{\substack{p\leq X\\p~\not\in~S}}\left(X^2\frac{\rho_f(p)}{p}+O(\rho_f(p))\right)=X^2\log \log X+O(X^2),
\end{align*}
where $\rho_f(p)=\{\bar{r}\in \Z/p\Z:f(\bar{r})=0\pmod {p}\}$, and the last equality follows from \cite[Cor.~3.2.2]{Murty-Alina}. Furthermore, we can run the summation above for $p$ up to at most $X$ because $h(f(r))=X^{O(1)}$, and there are only $O(1)$ many primes $p$ for which $\nu_p(f(r))\neq 0$, and $O(1)\leq \deg(f)$.

\subsubsection*{Case 2: estimate of the second sum in (\ref{eqn:sumwf})} For a rational $r$, denote $d(r)$ to be the denominator of $r$ in the lowest term. Then, we have
\begin{equation}\label{eqn:wdr}
\sum\limits_{h(r)\leq X}~\sum_{\substack{p~\mathrm{prime}~\not\in~S\\r_p=\infty}} 1=\sum_{h(r)\leq X}\omega_S(d(r))=X^2\log \log X+O(X^2),
\end{equation}
where the last equality follows by the proof of Corollary 3.2.2 in \cite{Murty-Alina}.
\subsubsection*{Case 3: estimate of the second sum in (\ref{eqn:hmm})} Let us write 
$$\sum_{h(r)\leq X}~ \sum_{\substack{p~\mathrm{prime}~\not\in~S\\\nu_p(f(r))\neq 0\\\nu_p(f(r))=\mathrm{even}}} 1=\sum_{h(r)\leq X}~ \sum_{\substack{p~\mathrm{prime}~\not\in~S\\\nu_p(f(r))> 0\\\nu_p(f(r))=\mathrm{even}}} 1+~\sum_{h(r)\leq X}~ \sum_{\substack{p~\mathrm{prime}~\not\in~S\\\nu_p(f(r))<0\\\nu_p(f(r))=\mathrm{even}}} 1.$$
The second sum above is, of course, bounded by (\ref{eqn:wdr}). To bound the first sum, we use Lemma~\ref{lem:triv}, and get 
\begin{equation}\label{eqn:even}
    \sum_{h(r)\leq X}~ \sum_{\substack{p~\mathrm{prime}~\not\in~S\\\nu_p(f(r))> 0\\\nu_p(f(r))=\mathrm{even}}} 1\leq \sum_{p\leq X}\left(X^2\frac{\rho_f(p^2)}{p^2}+O(\rho_f(p^2))\right)=O(X^2),
\end{equation}
where $\rho_f(p^2)=\{\bar{r}\in \Z/p^2\Z:f(\bar{r})=0\pmod {p^2}\}$, and this is at most $\mathrm{deg}(f)^2$, due to Lemma~\ref{lem:triv}.
Therefore, combining (\ref{eqn:hmm}), (\ref{eqn:sumwf}), (\ref{eqn:wdr}), and (\ref{eqn:even}), we get $\sum\limits_{h(r)\leq X} \omega_S(s(f(r)))=X^2\log\log X+O(X^2)$. To complete the proof of the lemma, it is now enough to show that 
\begin{equation}\label{eqn:sumsq}
\sum\limits_{h(r)\leq X} \omega_S(s(f(r)))^2\leq (X\log \log X)^2+O(X^2)
\end{equation}
Since $\omega(f(r))\leq \sum\limits_{\substack{p~\mathrm{prime}\\\nu_p(f(r))>0}} 1+\sum\limits_{\substack{p~\mathrm{prime}\\r_p=\infty}}1$, we get, $\sum\limits_{h(r)\leq X} \omega_S(s(f(r)))^2\leq  \sum\limits_{h(r)\leq X} \omega(f(r))^2\leq 2(S_1+S_2)$,   
where 
\begin{equation}\label{eqn:s1s2}
S_1=\sum\limits_{h(r)\leq X} \Big(\sum\limits_{\substack{p~\mathrm{prime}\\\nu_p(f(r))>0}} 1\Big)^2,~\mathrm{and}~S_2=\sum\limits_{h(r)\leq X} \Big(\sum\limits_{\substack{p~\mathrm{prime}\\r_p=\infty}}1\Big)^2.
\end{equation}
For any suitable parameter $Y\leq X$, applying Chinese remainder theorem, we have
\begin{align*}
&S_1(Y)=\sum\limits_{h(r)\leq X} \Big(\sum\limits_{\substack{p~\mathrm{prime}~\leq Y\\\nu_p(f(r))>0}} 1\Big)^2=X\Big(\sum_{\substack{p\neq q~\mathrm{primes}\\p,q\leq Y}} X\frac{\rho_f(p)\rho_f(q)}{pq}+O(1)\Big)=(X\log \log Y)^2+O(XY^2),\\
&S_2(Y)=\sum\limits_{h(r)\leq X} \Big(\sum\limits_{\substack{p~\mathrm{prime}~\leq Y\\r_p=\infty}}1\Big)^2=X\Big(\sum_{\substack{p\neq q~\mathrm{primes}\\p,q\leq Y}} X\frac{1}{pq}+O(1)\Big)=(X\log \log Y)^2+O(XY^2).
\end{align*}
The proof of equation (\ref{eqn:sumsq}), and consequently the proof of the lemma is now complete taking $Y=X^{1/2}$, as we have
$|S_1(Y)-S_1|,~|S_2(Y)-S_2|=O(X^2).$
\end{proof}

\subsubsection*{Proof of Proposition~\ref{prop:maintool}}
Denote $S_M(X)$ to be the number of $r\in \Q$ such that $\omega(s(f(r)))\leq M$, and $h(r)\leq X$. It follows from Lemma~\ref{lem:average} that $S_M(X)=O_M\left( \frac{X^2}{\log \log X}\right)$. This completes the proof of the proposition. \qed

\begin{remark}\rm
Note that Proposition~\ref{prop:maintool} addresses the number of prime factors in the square-free parts of the values of $f$. In contrast, \cite{Poonen} focuses on square-free values. When we restrict the evaluation of $f$ to integers, \cite[Thm.~3.2]{Poonen} shows that the proportion of square-free values of $f$ is precisely $\prod\limits_{p,~\mathrm{prime}}\left(1-\frac{c_p}{p^{2}}\right)$, where the quantity $c_p$ is studied in Lemma~\ref{lem:triv}. In particular, the quantity $\prod\limits_{p,~\mathrm{prime}}\left(1-\frac{c_p}{p^{2}}\right)$, is positive. In other words, if $f(t) \in \Z[t]$ is irreducible, then $f(n)$ is square-free for a positive proportion of the times.
\end{remark}
Finally, let us address the last preparatory counting tool, that is needed to prove the main result of this section.
\begin{lemma}\label{lem:bound}
    For any $\ell\in \{3,5,7\}$, and any integer $M$, we have the following estimate:
    $$\#\Big\{(A,B)\in S_{f_{\ell},g_{\ell}}(X): \omega(N(E_{A,B}))\leq M\Big\}=O_{M}\left(\frac{\#\mathcal{E}_{\ell}(X)}{\log \log X}\right).$$
\end{lemma}
\begin{proof}
Consider an elliptic curve $E_{A,B}$ in the minimal form. Then the conductor $E_{A,B}$ is (up to some bounded factor of $2$ and $3$) the product over all primes
$p$ dividing the discriminant $|\Delta(E_{A,B})|=|4A^3+27B^2|$. Therefore, it suffices to estimate the quantity, $\{(A,B)\in S_{f,g}(X): \omega(4A^3+27B^2)\leq M+2\}$. Now Lemma~\ref{lem:veryimp} implies that it is enough to estimate for each $\ell \in \{3,5,7\}$,
$$\Big\{r=a/b:|a|\leq X^{m_{\ell}},~|b|\leq X^{n_{\ell}}~\omega(4s(f(r)^3+27g(r)^2))\leq M+2\Big\}.$$
Denoting $\Delta(r) = 4f(r)^3 + 27g(r)^2$, it follows from Lemma~\ref{lem:irreducible} that there exists an irreducible factor $F(t)$ of $\Delta(t)$ that divides $\Delta(t)$ with a valuation of exactly $1$. We then have:
\begin{align*}
\#\Big\{r=a/b:|a|\leq X^{m_{\ell}},|b|\leq X^{n_{\ell}},~&\omega(s(\Delta(r)))\leq M+2\Big\}\\
&\leq \#\Big\{r=a/b:|a|\leq X^{m_{\ell}},|b|\leq X^{n_{\ell}},~\omega(F(r))\leq M+2\Big\}\\
&=O_{M}\left(\frac{X^{\max\{m_{\ell},n_{\ell}\}}}{\log \log X}\right)=O_{M}\left(\frac{\#\mathcal{E}_{\ell}(X)}{\log \log X}\right),
\end{align*}
where the last estimate follows from Proposition~\ref{prop:maintool}, and the fact that $\#\mathcal{E}_{\ell}(X)\asymp X^{m_{\ell}+n_{\ell}}$.
\end{proof} 
With this groundwork laid, we are now prepared to prove the main result of this section.
\subsection*{Proof of Theorem~\ref{thm:isogeny}}
The proof follows combining Lemma~\ref{lem:bound} and Lemma~\ref{lem:rankbound}.
\qed
\begin{remark}\rm
Additionally, Lemma~\ref{lem:bound} can be utilized to investigate the validity of $M$-Watkins over the fibers of certain elliptic surfaces over $\Q(t)$. This extends the work in \cite[Thm.~8]{BKP23}, where the authors examined the frequency of occurrences of even modular degrees over any such family.
\end{remark}

\bibliographystyle{amsplain} 
	\bibliography{ref.bib}

\end{document}